\documentclass[letterpaper,11pt]{amsart}
\usepackage{amsmath,amssymb,amsthm,pinlabel,tikz,hyperref,mathrsfs,color, thmtools}
\usepackage{verbatim}
\usepackage{float}
\usepackage{caption}
\usepackage{subcaption}
\usepackage{enumitem}
\newcommand{\nc}{\newcommand}
\nc{\dmo}{\DeclareMathOperator}
\dmo{\ra}{\rightarrow}
\dmo{\Prob}{\mathbb{P}}
\dmo{\E}{\mathbb{E}}
\dmo{\N}{\mathbb{N}}
\dmo{\Z}{\mathbb{Z}}
\dmo{\inte}{int}
\dmo{\Q}{\mathbb{Q}}
\dmo{\R}{\mathbb{R}}
\dmo{\C}{\mathcal{C}}
\dmo{\X}{\mathcal{X}}
\dmo{\U}{\mathcal{U}}
\dmo{\T}{\mathcal{T}}
\dmo{\F}{\mathcal{F}}
\dmo{\AC}{\mathcal{AC}}
\dmo{\w}{\omega}
\dmo{\MIN}{\mathcal{MIN}}
\dmo{\Mod}{Mod}
\dmo{\PMod}{PMod}
\dmo{\PMF}{\mathcal{PMF}}
\dmo{\Mat}{Mat}
\dmo{\supp}{supp}
\dmo{\UE}{\mathcal{UE}}
\dmo{\vol}{vol}
\dmo{\B}{B}
\dmo{\PB}{PB}
\dmo{\PR}{PSL(2,\mathbb{R})}
\dmo{\GL}{GL(k, \mathbb{C})}
\dmo{\SL}{SL(2, \mathbb{Z})}
\dmo{\Isom}{Isom}
\dmo{\RP}{\mathbb{R} \mathrm{P}}
\dmo{\I}{\mathcal{I}}
\dmo{\el}{\ell_{\C}}
\dmo{\NN}{\mathcal{N}}
\dmo{\rk}{rank}
\dmo{\tr}{tr}
\dmo{\llangle}{\langle\langle}
\dmo{\rrangle}{\rangle\rangle}
\dmo{\Unif}{Unif}
\dmo{\Out}{Out}
\dmo{\Homeo}{Homeo}
\dmo{\C*}{\mathcal{C}^{\dagger}}
\dmo{\sumRho}{\mathcal{N}}
\dmo{\stopping}{\vartheta}
\dmo{\diam}{\operatorname{diam}}

\usetikzlibrary{decorations.markings, patterns}
\tikzset{->-/.style={decoration={
  markings,
  mark=at position #1 with {\arrow{>}}},postaction={decorate}}}

\nc{\nt}{\newtheorem}

\nt{theorem}{Theorem}

\newtheorem{thm}{{\bf Theorem}}[section]
\newtheorem{defn}[thm]{{\bf Definition}}

\newtheorem{cor}[thm]{{\bf Corollary}}
\newtheorem{prop}[thm]{{\bf Proposition}}
\newtheorem{fact}[thm]{Fact}

\newtheorem{claim}[thm]{Claim}

\newtheorem*{ques}{Question}

\newtheorem{dfn}[thm]{Definition}
\numberwithin{equation}{section}
\newtheorem{obs}[thm]{Observation}

\title[WPD of some Homeos]{Metric weak proper discontinuity for pseudo-Anosov maps}

\date{\today}
\author{Inhyeok Choi}

\address{%
		Cornell University\\
		310 Malott Hall, Ithaca, NY, USA \\ \newline
		June E Huh Center for Mathematical Challenges, KIAS\\
		85 Hoegi-ro, Dongdaemun-gu, Seoul 02455, South Korea
}
\email{
        inhyeokchoi48@gmail.com
        }
\begin{document}
\begin{abstract}
We present a metric version of weak proper discontinuity (WPD) for pseudo-Anosov maps on surfaces. As an application, we show that there are plenty of quasi-morphisms on the homeomorphism group of a torus or a hyperbolic surface that are unbounded on certain homeomorphisms with large fixed set.


\noindent{\bf MSC classes:} 20F67, 30F60, 57K20, 57M60, 60G50
\end{abstract}

\maketitle

%
%

\section{Introduction}

This paper concerns the homeomorphism group $\Homeo(S)$ of a closed orientable surface $S$ of genus at least 1. Being a topological group with countably infinitely many connected components, $\Homeo(S)$ can be studied at the level of the identity component $\Homeo_{0}(S)$, the group of nullhomotopic homeomorphisms or at the level of the quotient group $\Mod(S)$, the mapping class group.

At the end of the 20th century, H. Masur and Y. Minsky proposed a new theory of the curve graph $\mathcal{C}(S)$, which shed light on new aspects of the mapping class group $\Mod(S)$ \cite{masur1999curve}. While mapping class groups are not hyperbolic nor CAT(0) in general, Masur and Minsky proved that $\mathcal{C}(S)$ is Gromov hyperbolic and $\Mod(S)$ acts non-elementarily on $\mathcal{C}(S)$. More specifically, they showed that pseudo-Anosovs act as a loxodromic on $\mathcal{C}(S)$. Although $\mathcal{C}(S)$ is not locally compact and the $\Mod(S)$-action on $\mathcal{C}(S)$ is not proper, this action does imply several geometric features of $\Mod(S)$ that is shared with (relatively) hyperbolic groups, cf. \cite{masur2000curve}.

For instance, M. Bestvina and K. Fujiwara \cite{bestvina2002bounded} observed that pseudo-Anosov mapping classes have the \emph{weak proper discontinuity} (WPD property). Roughly speaking, given a group $G$ acting on a metric space $X$, $g \in G$ is said to have the WPD property if the joint coarse stabilizer of distant $g$-orbit points is a finite set.  Using the WPD property of pseudo-Anosovs, Bestvina and Fujiwara constructed an infinite family of non-quasi-invertible pseudo-Anosovs in $\Mod(S)$, which in turn implies that the space $\widetilde{QH}(\Mod(S))$ of quasi-morphisms on $\Mod(S)$ (modulo bounded quasi-morphisms and homomorphisms) is infinite dimensional. This feature is shared with free groups \cite{MR624804} and hyperbolic groups \cite{MR1452851}. Bestvina-Fujiwara's WPD property of $\Mod(S)$-action on $\mathcal{C}(S)$ was strengthened into acylindricity by B. Bowditch \cite{bowditch2008tight}.

In \cite{bowden2022quasi}, J. Bowden, S. Hensel and R. Webb initiated the study of the \emph{fine curve graph} $\C*(S)$ of a surface $S$, an analogue of the classical curve graph $\mathcal{C}(S)$, for $\Homeo(S)$. By comparing the fine curve graph and the curve graph of punctured surfaces $S \setminus P$ for finite sets $P \subseteq S$, the authors found a non-quasi-invertible homeomorphism in $\Homeo_{0}(S)$ with respect to the action on $\C*(S)$. Using this fact and Bestvina-Fujiwara's theory \cite{bestvina2002bounded}, the authors deduced the infinite dimensionality of $\widetilde{QH}(\Homeo_{0}(S))$.

In \cite{bowden2024towards}, the three authors further investigated the dynamics of pseudo-Anosov maps relative to a finite set $P \subseteq S$. As a result, the authors proved that pseudo-Anosov maps  that are non-quasi-invertible at the level of mapping class (in $\Mod(S \setminus P)$) are non-quasi-invertible in $\Homeo(S)$ as well.

In this note, we make an observation following the theory of Bowden, Hensel and Webb. Note that $\Homeo_{0}(S)$ cannot have a loxodromic isometry of $\C*(S)$ with the WPD property, as it does not have nontrivial normal subgroups. A geometric reason for the absence of the WPD property is that the joint coarse stabilizer of any finite number of curves  $x_{1}, \ldots, x_{n}$ is an open (infinite) set in $\Homeo_{0}(S)$: a small perturbation of a map $\varphi$ leads to small perturbation of $\varphi x_{i}$. Nonetheless, one can hope that a joint coarse stabilizer of far enough $\varphi$-orbit points consist of finitely many translates of ``small" maps. We present such a ``metric" version of the WPD property of pseudo-Anosov maps.

\begin{defn}
Let $S$ be a closed orientable surface of genus at least 1, equipped with a metric $d$. We say that a homeomorphism $\varphi \in \Homeo(S)$ is an $\epsilon$-coarse identity of $S$ if $d(p, \varphi p) < \epsilon$ for each $p \in S$.
\end{defn}

\begin{thm}\label{thm:main}
Let $S$ be a closed orientable surface of genus at least 1 with metric $d$, and  let $\varphi \in \Homeo(S)$ be a homeomorphism. Suppose that there exists a $\varphi$-invariant finite set $P \subseteq S$ such that $\varphi|_{S \setminus P}$ is a pseudo-Anosov homeomorphism on $S \setminus P$. Then for each $\epsilon > 0$, for each simple closed curve $x_{0} \in \C*(S)$ and for each $L>0$, there exists $N, M>0$ and a finite set of homeomorphisms $\{h_{1}, \ldots, h_{N}\}$ in $\Homeo(S)$ such that \[
\begin{aligned}
&\big\{ g \in \Homeo(S) : d_{\C*}(\varphi^{-N}x_{0}, g\varphi^{-N}x_{0}) < L\,\, \textrm{and} \,\,d_{\C*}(\varphi^{N} x_{0}, g \varphi^{N} x_{0}) < L \big\} \\
&\subseteq \big\{ hh_{i} : i=1, \ldots, M, \textrm{$h$ is an $\epsilon$-coarse identity of $S$}\big\}.
\end{aligned}
\]
\end{thm}

As a corollary, we construct quasi-morphisms associated to certain loxodromic isometries of $\C*(S)$ that are \emph{not} pseudo-Anosov maps.

\begin{thm}\label{thm:scl}
Let $S$ be a surface of genus at least $1$. Then there exists a homeomorphism $\varphi \in \Homeo_{0}(S)$ whose fixed point set has nonempty interior and whose stable commutator length (in $\Homeo(S)$) is nonzero.
\end{thm}

\begin{thm}\label{thm:inf}
Let $S$ be a close orientable surface of genus at least $1$ and let $D$ be an embedded closed disc on $S$. Let $Fix_{D} \le \Homeo(S)$ be the group of homeomorphisms that preserve $D$ pointwise. Then the space $\widetilde{QH}\big(\Homeo(S); Fix_{D} \cap \Homeo_{0}(S)\big)$ of quasi-morphisms on $\Homeo(S)$ that is unbounded on $Fix_{D} \cap \Homeo_{0}(S)$ is infinite dimensional.
\end{thm}

There is another cute application of Theorem \ref{thm:main} for random walks on $\Homeo(S)$. Let $\mu$ be a discrete probability measure on $\Homeo(S)$, i.e., $\supp \mu$ is a countable subset of $\Homeo(S)$. By convoluting $\mu$ from the right, one can run a (right) $\mu$-random walk on $\Homeo(S)$ and observe it on $\C*(S)$. If the subsemigroup $\llangle \supp \mu \rrangle$ generated by $\supp \mu$ contains two independent loxodromic isometries of $\C*(S)$, then the almost every sample path on $\C*(S)$ converges to a boundary point in $\partial \C*(S)$. This gives the hitting measure $\nu$ on $\partial \C*(S)$. One can then ask if $(\partial \C*(S), \nu)$ captures almost all measurable information of $\mu^{\ast \Z_{>0}}$. In other words, we can ask if $(\partial \C*(S), \nu)$ is the \emph{Poisson boundary} for $(\Homeo(S), \mu)$.

This sort of question has rich history; see \cite{kaimanovich1994poisson}, \cite{kaimanovich1996poisson} and \cite{maher2018random} for references. This question is usually answered with some conditions on $\mu$. For example, one usually assumes that $\mu$ has finite first moment (with respect to the hyperbolic metric--the fine curve graph distance in this case), finite logarithmic moment or finite entropy (which is purely about the weight distribution of $\mu$ and does not involve the $\C*(S)$-distance). In this direction, one of the most general answers is as follows:

\begin{prop}[{\cite[Theorem 1.2]{chawla2022the-poisson}}]
Let $G$ be a countable group with a non-elementary action on a geodesic hyperbolic space $X$ with at least one WPD element. Let $\mu$ be a generating probability measure
on $G$ with finite entropy. Let $\partial X$ be the hyperbolic boundary of $X$, and let $\nu$ be the hitting measure of the random walk driven by $\mu$. Then the space $(\partial X, \nu)$ is the Poisson boundary of $(G, \mu)$.
\end{prop}

Since the fine curve graph $\C*(S)$ is Gromov hyperbolic (\cite[Theorem 3.10]{bowden2022quasi}), we can apply this proposition. Let us say that a subgroup $G \le \Homeo(S)$ is \emph{locally discrete} if it is discrete in $C^{0}$-topology. Our Theorem \ref{thm:main} implies that pseudo-Anosov maps are WPD (in the classical sense) \emph{inside discrete subgroups of $\Homeo(S)$}. Hence, we have:

\begin{cor}[{\cite[Theorem 1.2]{chawla2022the-poisson}}]
Let $S$ be a closed orientable surface of genus $\ge 1$ and let $G$ be a non-virtually cyclic, locally discrete countable subgroup of $\Homeo(S)$ that contains a pseudo-Anosov map $\varphi$ (relative to a finite subset of $S$).  Let $\mu$ be a generating probability measure on $G$ with finite entropy. Let $\nu$ be the hitting measure on $\partial \C*(S)$ of the random walk driven by $\mu$. Then the space $(\partial \C*(S), \nu)$ is the Poisson boundary of $(G, \mu)$.
\end{cor}

While preparing this paper, the author was made aware of an independent work of S. Hensel and F. Le Roux \cite{hensel2025rotation}. In the case of the torus $\mathbb{T}^{2}$,  they have proved that elements of $\Homeo_{0}(\mathbb{T}^{2})$ that act loxodromically on $\C*(\mathbb{T}^{2})$ have the metric WPD property. This is more general than the result of Theorem \ref{thm:main}. Indeed,  loxodromic isometries of $\Homeo_{0}(\mathbb{T}^{2})$ include not only pseudo-Anosov maps on $\mathbb{T}^{2}$, but also maps homotopic to a pseudo-Anosov map with respect to the singularity set of the pA map. Meanwhile, we deal with higher genus case as well as the torus. Our method and their method are different; their approach builds upon the theory of rotation sets. Since there are higher genus version of rotation sets studied by P-A. Guih{\'e}neuf and E. Militon \cite{guiheneuf2023hyperbolic}, it is natural to ask the following question:

\begin{ques}
For $g \ge 2$, does every loxodromic isometry of $\C*(S_{g})$ have metric WPD property?
\end{ques}

\subsection*{Acknowledgement}

The author thanks Sebastian Hensel and Kathryn Mann for the discussion about the dynamics of pseudo-Anosov maps on the fine curve graph. The author is also grateful to Sebastian Hensel and Fr{\'e}d{\'e}ric Le Roux for personally sharing their manuscript.

Part of this work was done while the author was visiting June E Huh Center for Mathematical Challenges at KIAS in June 2025.

\section{Preliminaries}

Throughout the paper, $S$ denotes a closed orientable surface of genus $g \ge 1$. We recall the following notion of Bowden, Hensel and Webb (\cite[Definition 2.5]{bowden2024towards}): a bifoliated singular flat (BSF) structure on $S$ is a tuple $q=(\mathcal{F}_{v}, \mathcal{F}_{h},d)$ of a singular flat metric $d$, all of whose cone angles are multiples of $\pi$, and transverse (singular) foliations $\mathcal{F}_{v}, \mathcal{F}_{h}$ which are locally geodesic and orthogonal for $d$ away from the cone points. 

In this paper, the following fact is relevant. Let $P$ be a finite subset of $S$ and let $\varphi$ be a pseudo-Anosov map $\varphi$ on $S \setminus P$. Then there exists a BSF structure $q$ on $S$ that comes from the unstable and the stable measured foliations for $\varphi$. The two transverse foliations for $q$ are \emph{ending}, i.e., every pseudo-leaf is simply-connected and contains at most 1 singularity point.

Let $S$ be a surface with a fixed metric $d$, often coming from a BSF structure. An \emph{$\epsilon$-coarse stabilizer} is a homeomorphism $\varphi \in \Homeo(S)$ such that $d(p, \varphi p) < \epsilon$ for each $p \in S$.

For a point $x \in \C*(S)$ and $L>0$, we say that a homeomorphism $\varphi \in \Homeo(S)$ \emph{$L$-stabilizes} $x$ if $d_{\C*(S)}(x, \varphi x) < L$. For $x, y \in \C*(S)$ and $L>0$ we define the $L$-joint stabilizer \[
Stab_{L}(x, y) := \{ \varphi \in \Homeo(S) : d_{\C*(S)} (x, \varphi x) < L \wedge d_{\C*(S)} (y, \varphi y) < L\}.
\]

Bowden, Hensel and Webb studied in \cite{bowden2024towards} the dynamics of a pseudo-Anosov map $\varphi$ on $\C*(S)$, providing a general toolkit to describe the convergence to the forward/backward fixed points of $\varphi$ in $\partial \C*(S)$. We gather some of their results here.

\begin{thm}[{Target lemma, \cite[Lemma 3.6]{bowden2024towards}}]\label{thm:target}
Let $q$ be a BSF structure on $S$ with ending vertical foliation. Then for any $B_{target}, \epsilon_{target}>0$ there exists $L_{target}>0$ such that the following holds.

Suppose that $\gamma : [0, 1] \rightarrow S$ is a path so that the total width of $\gamma$ is at least $\epsilon_{target}$ and the size of $\gamma$ is at most $B_{target}$. Then any vertical geodesic for $q$ of length at least $L_{target}$ intersects $\gamma$.
\end{thm}

\begin{thm}[{Escape lemma, \cite[Lemma 4.7]{bowden2024towards}}]\label{thm:escape}
Let $q$ be a BSF structure on $S$ and let $x_{0} \in \C*(S)$ be a simple closed curve on $S$. Then for any $L_{escape}>0$ there exists $R_{escape}>0$ such that, if $\beta$ is a simple closed curve of ($q$-)size at most $L_{escape}$, then $d_{\C*}(\beta, x_{0}) \le R_{escape}$.
\end{thm}

\begin{thm}[{Small width lemma, \cite[Lemma 4.6]{bowden2024towards}}]\label{thm:small}
Let $q$ be a BSF structure on $S$ with ending vertical foliation $\mathcal{F}_{v}$. Let $x_{0} \in \C*(S)$ be a simple closed curve. Then for any $B_{small}, \epsilon_{small}>0$ there exists $K_{small}>0$ such that the following holds.

Let $\beta$ be a simple closed curve such that $(\beta, \mathcal{F}_{v})_{x_{0}} > K_{small}$. Then every subsegment of $\beta$ of size $B_{small}$ has total horizontal width at most $\epsilon_{small}$, and thus $\epsilon_{small}$-fellow travels a (possibly singular) leaf of $\mathcal{F}_{v}$.
\end{thm}

\begin{dfn}\label{dfn:strip}
Let $q$ be a BSF structure on $S$. We say that a region $R \subseteq S$ is a \emph{strip} if it is an isometric embedding of the Euclidean rectangle into $S$ respecting horizontal and vertical directions. It is called \emph{horizontal} if the width is larger than the height, and \emph{vertical} if not.

Let $\mathcal{S}$ be a strip in $S$. We say that a path $\gamma$ \emph{traverses $\mathcal{S}$ horizontally} if $\gamma$ connects the two vertical sides while staying inside $\inte \mathcal{S}$ during the journey. We similarly define paths transversing $\mathcal{S}$ vertically.
\end{dfn}

\begin{dfn}\label{dfn:quasi}
We say that two paths $\alpha$ and $\beta$ on $S$ are \emph{quasi-transverse} if they have finitely many intersections. 
\end{dfn}

\begin{dfn}\label{dfn:bifoliated}
Let $q$ be a BSF structure on $S$. A \emph{bifoliated square} $\mathcal{U}$ on $S$ is the image of an isometric embedding $i : (-\epsilon, \epsilon) \times (-\epsilon, \epsilon) \rightarrow S$ for some $\epsilon > 0$ that respects the horizontal and standard directions. In this case, $i(0, 0)$ is called the \emph{center} of $\mathcal{U}$ and $\mathcal{U}$ is said to have the side length $2\epsilon$. Furthermore, $\mathcal{U}$ is called a \emph{bifoliated standard neighborhood} of $i(0, 0)$.
\end{dfn}
Note that a bifoliated square with side length $2\epsilon$ has diameter $\le 4\epsilon$.

In general, we can think of a \emph{bifoliated (regular) $2k$-gon} $\mathcal{U}$ on $S$ for $k \ge 1$. It is straightforward to define the center and the side length of $\mathcal{U}$. We say that $\mathcal{U}$ is a \emph{bifoliated standard neighborhood} of the center of $\mathcal{U}$.

When $S$ is given a BSF structure, every point $p$ of $S$ admits a bifoliated standard neighborhood. If $p$ is a $k$-fold singularity, then such a neighborhood is a bifoliated $2k$-gon.

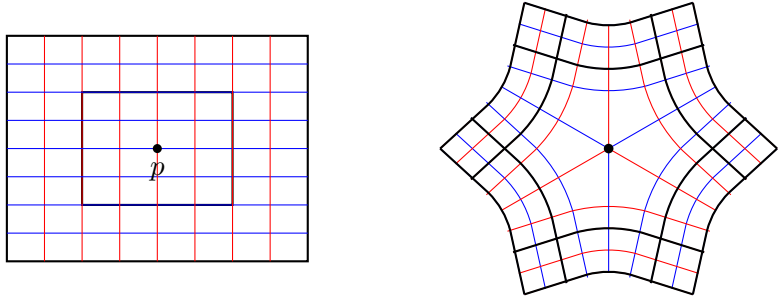
\begin{figure}
\begin{tikzpicture}
\def\c{1}
\draw[thick] (0, 0) -- (4*\c, 0) -- (4*\c, 3*\c) -- (0, 3*\c) -- cycle;
\draw[thick] (1*\c, 0.75*\c) -- (3*\c, 0.75*\c) -- (3*\c, 2.25*\c) -- (\c, 2.25*\c) -- cycle;
\foreach \i in {1, ..., 7}{
\draw[color=blue] (0, 3*\i*\c/8) -- (4*\c, 3*\i*\c/8);
\draw[color= red] (4*\i*\c/8, 0) -- (4*\i*\c/8, 3*\c);
}
\fill (2*\c, 1.5*\c) circle (0.06);
\draw (2*\c, 1.2*\c) node {$p$};

\begin{scope}[shift={(8*\c, 1.5*\c)}, scale=0.8]

\foreach \j in {0, 1, 2}{
\foreach \i in {0, ..., 3}{
\draw[rotate=120*\j, color=blue] ({-2.8*\c*cos(60) - 0.35*\c*\i*tan(15)}, {2.8*\c*sin(60) - 0.35*\c*\i}) -- ({0 - 1.4*\c*sin(15)- 0.35*\c*\i*tan(15)}, {2.8*\c*sin(60) - 2.8*\c*cos(60)*tan(15) + 1.4*\c - 1.4*\c*cos(15) - 0.35*\c*\i}) arc (-105:-75:{1.4*\c + 0.35*\i /cos(15)})  -- ({2.8*\c*cos(60)+0.35*\c*\i*tan(15)}, {2.8*\c*sin(60) - 0.35*\c*\i});
}
\draw[rotate=120*\j, color=red] (0, 0) -- (0, {2.8*\c*sin(60) - 2.8*\c*cos(60)*tan(15) });
}

\foreach \j in {0, 1, 2}{
\foreach \i in {0, ..., 3}{
\draw[rotate=120*\j + 60, color=red] ({-2.8*\c*cos(60) - 0.35*\c*\i*tan(15)}, {2.8*\c*sin(60) - 0.35*\c*\i}) -- ({0 - 1.4*\c*sin(15)- 0.35*\c*\i*tan(15)}, {2.8*\c*sin(60) - 2.8*\c*cos(60)*tan(15) + 1.4*\c - 1.4*\c*cos(15) - 0.35*\c*\i}) arc (-105:-75:{1.4*\c + 0.35*\i /cos(15)})  -- ({2.8*\c*cos(60)+0.35*\c*\i*tan(15)}, {2.8*\c*sin(60) - 0.35*\c*\i});
}
\draw[rotate=120*\j+60, color=blue] (0, 0) -- (0, {2.8*\c*sin(60) - 2.8*\c*cos(60)*tan(15) });
}

\fill (0,0) circle (0.08);

\foreach \j in {0, ..., 5}{
\foreach \i in {0, 2}{
\draw[rotate=60*\j, thick] ({-2.8*\c*cos(60) - 0.35*\c*\i*tan(15)}, {2.8*\c*sin(60) - 0.35*\c*\i}) -- ({0 - 1.4*\c*sin(15)- 0.35*\c*\i*tan(15)}, {2.8*\c*sin(60) - 2.8*\c*cos(60)*tan(15) + 1.4*\c - 1.4*\c*cos(15) - 0.35*\c*\i}) arc (-105:-75:{1.4*\c + 0.35*\i /cos(15)})  -- ({2.8*\c*cos(60)+0.35*\c*\i*tan(15)}, {2.8*\c*sin(60) - 0.35*\c*\i});
}
}

\end{scope}
\end{tikzpicture}

\caption{Bifoliated regular $2k$-gons for $k=2$ and $k=3$.}
\label{fig:bifold}
\end{figure}

\section{Proof}

We will in fact prove the following variant of Theorem \ref{thm:main}:

\begin{thm}\label{thm:mainVar}
Let $S$ be a surface of genus at least 1 and let $\varphi \in \Homeo(S)$ be a homeomorphism. Suppose that there exists a $\varphi$-invariant finite set $P \subseteq S$ such that $\varphi|_{S \setminus P}$ is a pseudo-Anosov homeomorphism on $S \setminus P$. Then for each $\epsilon > 0$, for each simple closed curve $x_{0} \in \C*(S)$ and $L>0$, there exists $N>0$ such that the following holds.

Let $(g_{1}, g_{2}, \ldots) \in \Homeo(S)$ be an infinite sequence of homeomorphisms that $L$-stabilize $\varphi^{-N} x_{0}$ and $\varphi^{N} x_{0}$. Then there exists $i \neq j$ such that $g_{i}g_{j}^{-1}$ is an $\epsilon$-coarse identity.
\end{thm}

\begin{proof}[Proof of Theorem \ref{thm:main} assuming Theorem \ref{thm:mainVar}]
For given $\varphi, \epsilon, x_{0}$ and $L$, let $N$ be the constant as in Theorem \ref{thm:mainVar}. We now construct a sequence in the joint $L$-stabilizer $Stab_{L}(\varphi^{-N} x_{0}, \varphi^{N} x_{0})$. We first pick arbitrary element in $Stab_{L}(\varphi^{-N} x_{0}, \varphi^{N} x_{0})$ as $g_{1}$. Having chosen $g_{1}, \ldots, g_{n}$, we pick $g_{n+1} \in Stab_{L}(\varphi^{-N} x_{0}, \varphi^{N} x_{0})$ that is outside $\cup_{i=1}^{n} \{\textrm{$\epsilon$-coarse identities}\} \cdot g_{i}$, if there is any. If there is no such element, we halt.

If this process continues indefinitely, then we obtain an infinite sequence $(g_{1}, g_{2}, \ldots)$ in  $Stab_{L}(\varphi^{-N} x_{0}, \varphi^{N} x_{0})$ such that $g_{i}g_{j}^{-1}$ is not an $\epsilon$-coarse identity for $i \neq j$. This contradicts Theorem \ref{thm:mainVar}. Hence the process must halt at some finite step. The desired conclusion of Theorem \ref{thm:main} follows.
\end{proof}

Let us now prove Theorem \ref{thm:mainVar}.

\begin{proof}[Proof of Theorem \ref{thm:mainVar}]
Suppose that $\varphi$, $\epsilon$ and $L$ are given. Let $\mathcal{F}_{h}$ and $\mathcal{F}_{v}$ be the horizontal and the vertical foliations, respectively, for $\varphi|_{S \setminus P}$. Let $q$ be the BSF structure for $\varphi$.

Because $S$ is equipped with the BSF structure $q$, every point $p \in S$ admits a bifoliated standard neighborhood $\mathcal{U} = \mathcal{U}_{p}$ of side length $\le \epsilon/2$. For simplicity, we will describe the situation where $p$ is not a singularity point. Then $\mathcal{U}_{p}$ is the image of an isometric embedding $i : (-\epsilon_{p}, \epsilon_{p}) \times  (-\epsilon_{p}, \epsilon_{p}) \rightarrow S$ for some $0<\epsilon_{p}\le \epsilon/2$. Let $\mathcal{V}_{p}$ be another bifoliated standard neighborhood of $p$ with side length $\epsilon_{p}$, i.e., $\mathcal{V}_{p}:= i( (-\epsilon_{p}/2, \epsilon_{p}/2)\times (-\epsilon_{p}/2, \epsilon_{p}/2))$. We now define \[\begin{aligned}
H^{(1)} (\mathcal{U}) := i \big( (-\epsilon_{p}, \epsilon_{p}) \times (-\epsilon_{p}, -\epsilon_{p}/2) \big), \quad 
H^{(2)} (\mathcal{U}) := i \big((-\epsilon_{p}, \epsilon_{p}) \times (\epsilon_{p}/2, \epsilon_{p}) \big), \\
V^{(1)} (\mathcal{U}) := i \big( (-\epsilon_{p}, -\epsilon_{p}/2) \times (-\epsilon_{p}, \epsilon_{p})  \big), \quad 
V^{(2)} (\mathcal{U}) := i \big((\epsilon_{p}/2, \epsilon_{p}) \times (-\epsilon_{p}, \epsilon_{p})   \big).
\end{aligned}
\]
We call the vertical segment $i\big(\{0\} \times [-7\epsilon_{p}/8, -5\epsilon_{p}/8]\big)$ the \emph{witness} for $H^{(1)}(\mathcal{U})$. We similarly define the witnesses for the other horizontal/vertical strips as the centered vertical/horizontal geodesic segment of length $\epsilon_{p}/4$.

The union of $H^{(1)}(\mathcal{U})$, $H^{(2)}(\mathcal{U})$, $V^{(1)}(\mathcal{U})$ and $V^{(2)}(\mathcal{U})$ is equal to $\mathcal{U} \setminus \bar{\mathcal{V}}$. Their interlaced pattern implies the following. Let $\gamma_{1}, \gamma_{2}, \eta_{1}, \eta_{2}$ be mutually quasi-transverse paths such that:
\begin{enumerate}
\item $\gamma_{1}$ and $\gamma_{2}$  traverse $H^{(1)}(\mathcal{U})$ and $H^{(2)}(\mathcal{U})$ horizontally, respectively, and 
\item  $\eta_{1}$ and $\eta_{2}$ traverse $V^{(1)} (\mathcal{U})$ and $V^{(2)}(\mathcal{U})$ vertically, respectively. 
\end{enumerate}
In particular, $\gamma_{1}, \ldots, \eta_{2}$ do not intersect with $\mathcal{V}$. In this case, the connected component $R$ of $\mathcal{V}$ in  $S \setminus (\gamma_{1} \cup \gamma_{2} \cup \eta_{1} \cup \eta_{2})$ is contained in $\mathcal{U}$. For instance, any path in $\mathcal{U}$ from $p$ to the bottom side of $H^{(1)}(\mathcal{U})$ must intersect $\gamma_{1}$ by the Jordan curve theorem. For a similar reason, any path in $\mathcal{U}$ from $p$ to $\partial \mathcal{U}$ should intersect $\gamma_{1} \cup \ldots \cup \eta_{2}$ and cannot lie in $R$.

Furthermore, among the two sides of $\gamma_{1}$, only one side is facing $p$. More precisely, there exists one side of $\gamma_{1}$ such that no path in $\mathcal{U}$ connects $p$ to that side without crossing $\gamma_{1}$. In other words, $R$ cannot approach a point on $\gamma_{1}$ from two sides, and $\partial R$ is embedded. Finally, recall that $\gamma_{1} \cup \gamma_{2}$ and $\eta_{1} \cup \eta_{2}$ have finitely many intersections. It is now clear that $R$ is a polygon and $\partial R$ is a simple closed curve on $S$ that is a  concatenation of finitely many subsegments of $\gamma_{1}, \ldots, \eta_{2}$.

When $p$ is a singularity point with multiplicity $k$, then $\mathcal{U} = \mathcal{U}_{p}$ is a bifoliated $2k$-gons; let $\epsilon_{p}$ be the half of the side length of $\mathcal{U}$. We can similarly scale down $\mathcal{U}$ by the factor of 2 to define  $\mathcal{V} = \mathcal{V}_{p}$. We analogously define the interlaced horizontal strips $H^{(1)}(\mathcal{U}), \ldots, H^{(k)}(\mathcal{U})$ and vertical strips $V^{(1)} (\mathcal{U}), \ldots, V^{(k)}(\mathcal{U})$, and their vertical/horizontal witnesses.

Since $S$ is compact, we can take a finite subcover $\mathcal{V}_{p_{1}}, \ldots, \mathcal{V}_{p_{T}}$ of $\{\mathcal{V}_{p} : p \in S\}$ for $S$, where $p_{1}, \ldots, p_{T} \in S$. Let \[
D := \sum_{l=1}^{T} \deg(p_{l}), \quad \epsilon_{min} := \min_{l=1, \ldots, T} \epsilon_{p_{l}} > 0.
\]
 For convenience, we write $\mathcal{V}_{p_{l}}$ as $\mathcal{V}_{l}$ and $\mathcal{U}_{p_{l}}$ as $\mathcal{U}_{l}$. 
 Let us also define \[
 \begin{aligned}
\mathcal{C}_{H} &:= \big\{ \textrm{horizontal strips associated to $\mathcal{U}_{1}, \ldots, \mathcal{U}_{T}$}\big\} \\
&= \big\{ \mathcal{H}^{(i)}(\mathcal{U}_{p_{l}}) : l=1,\ldots, T, i = 1, \ldots, (\textrm{multiplicity of $p_{l}$}) \big\}.
\end{aligned}
\]
For convenience, we will also enumerate $\mathcal{C}_{H}$ as \[
\mathcal{C}_{H} := \{ H_{1}, H_{2}, \ldots, H_{D}\}.
\]
Similarly we define the collection $\mathcal{C}_{V} = \{V_{1}, \ldots, V_{D}\}$ of the vertical strips associated to $\mathcal{U}_{1}, \ldots, \mathcal{U}_{T}$.

Since $\epsilon_{min} \le \epsilon_{p_{l}} < \epsilon/2$ for $l=1, \ldots, T$, the following is straightforward.
\begin{obs}
\label{obs:horizontalPass}Fix a horizontal strip $H_{l} \in \mathcal{C}_{H}$ and let $\eta_{l}$ be its witness. Let $\Gamma$ be a $q$-horizontal geodesic, let $\Gamma'$ be the subsegment of $\Gamma$ after truncating $\epsilon$-long terminal subsegments from the both side, and let $\alpha$ be a path that is $0.1\epsilon_{min}$-fellow traveling with $\Gamma$. Suppose that  $\Gamma'$ intersects $\eta_{l}$. Then $\alpha$ traverses $H_{l}$ horizontally.
\end{obs}
An analogous statement holds for vertical strips in $\mathcal{C}_{V}$ as well.

Before proceeding to the proof, note that we can choose our favorite basepoint $x_{0} \in \C*(S)$, up to increasing $L$. In particular, we may first pick a non-nullhomotopic simple closed curve $x_{0}' \in \C*(S)$ and approximate it with another simple closed curve $x_{0}$ that has piecewise constant $q$-slope that is never vertical nor horizontal. (It is not hard to realize $d_{\C*}(x_{0}', x_{0}) = 1$.) Note that if a $q$-geodesic segment $\gamma$ has $q$-slope $\alpha$, then $\varphi(\gamma)$ has $q$-slope $\lambda^{2} \cdot \alpha$, where $\lambda$ is the stretch factor of $\varphi$. Hence, given our choice of $x_{0}$, for sufficiently large $k$, $\varphi^{k}x_{0}$ and $x_{0}$ have locally distinct slopes away from finitely many ``kinks". In particular, $\varphi^{k} x_{0}$ and $x_{0}$ are quasi-transverse and each complementary region of $\varphi^{k} x_{0} \cup x_{0}$ has finitely many sides.

Let us now start the proof. For sufficiently large $N$ that will be quantified by the subsequent claims, let $g_{1}, g_{2}, \ldots$ be elements in $Stab_{L}(\varphi^{-N}x_{0}, \varphi^{N} x_{0})$. Our choice of $x_{0}$ guarantees that: \begin{claim}\label{claim:quasiTrans}
For each sufficiently large $k$, $x_{0}$ and $\varphi^{k} x_{0}$ are quasi-transverse.
\end{claim}

Our next claim is: \begin{claim}\label{claim:horizontalPass}
For each sufficiently large $N$, for each $i$, $g_{i} \varphi^{N} x_{0}$ traverses each $V_{l} \in \mathcal{C}_{V}$ vertically.
\end{claim}

To show this, let $L'=L_{target}$ be as in Theorem \ref{thm:target} for $B_{target} = \epsilon_{target} := \epsilon_{min}/10$. Let $R'=R_{escape}$ be as in Theorem \ref{thm:escape} for $L_{escape} := L'+3\epsilon$.
Let $K'=K_{escape}$ be  as in  Theorem \ref{thm:small} for $B_{small} := L'+ 3\epsilon$ and $\epsilon_{small} := \epsilon/100$. Finally, let $N$ be large enough such that  such that \begin{equation}\label{eqn:NLarge}
\begin{aligned}
\big( \varphi^{N} x_{0}, \mathcal{F}_{v}\big)_{x_{0}} &> L+L' + R'+ K'+ 100\delta_{\C*},\\
\big(\varphi^{-N} x_{0}, \mathcal{F}_{h}\big)_{x_{0}} &> L+L' + R'+ K'+ 100\delta_{\C*}.
\end{aligned}
\end{equation}
(Recall that $\mathcal{F}_{v}$ and $\mathcal{F}_{h}$ are the attractor/repeller for the action of $\varphi$ on $\C*(S)$.) Here, $\lambda_{\varphi}^{\C*}$ is the $d_{\C*}$-translation length of $\varphi$ and $\delta_{\C*}$ is the Gromov hyperbolicity constant for $\C*(S)$.

Let $i \in \Z_{>0}$. Since $d_{\C*}(\varphi^{N} x_{0}, g_{i} \varphi^{N} x_{0}) \le L$, Inequality \ref{eqn:NLarge} implies that \[
\big(g_{i} \varphi^{N} x_{0}, \mathcal{F}_{v}\big)_{x_{0}} \ge \big(\varphi^{N} x_{0}, \mathcal{F}_{v} \big)_{x_{0}} - L \ge (L + L'+K'+R') - L \ge K'+R'.
\]In particular, we have $d_{\C*}(g_{i} \varphi^{N} x_{0}, x_{0}) > R'$. Theorem \ref{thm:escape} tells us that the $q$-size of $g_{i}\varphi^{N} x_{0}$ is at least $L' + 3\epsilon$. Take any $(L' + 3\epsilon)$-long subsegment $\beta$ of $g_{i} \varphi^{N} x_{0}$. Theorem \ref{thm:small} then tells us that $\beta$ has width at most $\epsilon_{min}/100$. Let $\Lambda$ be a vertical segment in $q$ that is within Hausdorff distance $\epsilon_{min}/100$ from $\beta$. Then the length of $\Lambda$ is at least the size of $\beta$ minus the width of $\beta$, hence at least $L' +2\epsilon$. Let $\Lambda'$ be the subsegment of $\Lambda$ after truncating $\epsilon$-long terminal subsegments. Then $\Lambda'$ has length $L'$. Theorem \ref{thm:target} implies that $\Lambda'$ meets the witness for $V_{l}$, which is a horizontal segment longer than  $\epsilon_{min}/4$. A vertical version of Observation \ref{obs:horizontalPass} now guarantees that $\beta$ traverses $V_{l}$ vertically. Hence, $g_{i} \varphi^{N} x_{0}$ traverses $V_{l}$ vertically.

A similar argument using the ending property of $\mathcal{F}_{h}$ implies that: \begin{claim}\label{claim:verticalPass}
For each sufficiently large $N$, for each $i$, $g_{i} \varphi^{-N} x_{0}$ traverses each $H_{l} \in \mathcal{C}_{H}$ horizontally.
\end{claim}

Assuming that $N$ is large enough for Claim \ref{claim:quasiTrans}, \ref{claim:horizontalPass} and \ref{claim:verticalPass}, let $\mathcal{A}$ ($\mathcal{B}$, resp.) be the collection of subsegments of $\varphi^{-N} x_{0}$ ($\varphi^{N} x_{0}$, resp.) whose endpoints lie in $\varphi^{-N} x_{0} \cap \varphi^{N} x_{0}$. Since $\varphi^{-N}x_{0}$ and $\varphi^{N} x_{0}$ are quasi-transverse, $\mathcal{A}$ and $\mathcal{B}$ are finite sets.

Let us now pick $i \in \Z_{>0}$ and $l \in \{1, \ldots, T\}$. Claim \ref{claim:horizontalPass} and \ref{claim:verticalPass} tells us that $g_{i} \varphi^{N} x_{0}$ vertically traverses the vertical strips associated to $\mathcal{U}_{l}$ and $g_{i} \varphi^{-N} x_{0}$ horizontally traverses the horizontal strips associated to $\mathcal{U}_{l}$. Hence, there is an open region $R_{i, l}$ sandwiched between $\mathcal{V}_{l}$ and $\mathcal{U}_{l}$ such that $\partial R_{i, l}$ is a concatenation of subsegments of $g_{i} \varphi^{N} x_{0}$ and $g_{i} \varphi^{-N} x_{0}$.  For concreteness, let $\gamma_{1}, \ldots, \gamma_{N(i, l)}$ be the (distinct) subsegments of $g_{i} \varphi^{-N} x_{0}$ and $\eta_{1}, \ldots, \eta_{N(i, l)}$ be the (distinct) subsegments of $g_{i} \varphi^{N} x_{0}$ that together comprise $\partial R_{i, l}$. Then $\{g_{i}^{-1} \gamma_{1}, g_{i}^{-1} \eta_{1}, \ldots, g_{i}^{-1} \gamma_{N(i, l)}, g_{i}^{-1} \eta_{N(i, l)}\}$ are consecutive arcs on $S$ that form a boundary of an immersed disc. This gives a map \[
F_{l} : i \mapsto \big(g_{i}^{-1}\{ \gamma_{1}, \gamma_{2}, \ldots\}, g_{i}^{-1}\{\eta_{1}, \eta_{2}, \ldots\}\big) \in 2^{\mathcal{A}} \times 2^{\mathcal{B}}.
\]

Since $2^{\mathcal{A}} \times 2^{\mathcal{B}}$ and $\{1, \ldots, T\}$ are finite sets, there exist distinct integers $i \neq j$ with the same values of  $(F_{1}, \ldots, F_{T})$. We now claim that $g_{i} g_{j}^{-1}$ is an $\epsilon$-coarse identity. To show this, let $p \in S$ be an arbitrary point. Then $p \in \mathcal{V}_{l} \subseteq R_{j, l}$ for some $l \in \{1, \ldots, T\}$. Now $g_{j}^{-1}$ sends $\partial R_{j, l}$ to the union of the arcs in $F_{l}(j) = F_{l}(i)$. This is now mapped by $g_{i}$ to $\partial R_{i, l}$, the boundary of $R_{i, l}$. By Jordan curve theorem, $g_{i} g_{j}^{-1}$ then sends $R_{j, l}$ onto a component of $S \setminus \partial R_{i, l}$. Since the surface $S$ has genus at least 1, no two distinct embedded disc can share the boundary. Hence, we have  $g_{i} g_{j}^{-1}(R_{j, l}) = R_{i, l}$ and $g_{i} g_{j}^{-1}(p) \in R_{i, l} \subseteq \mathcal{U}_{l}$. Since $p \in \mathcal{V}_{l} \subseteq \mathcal{U}_{l}$ as well, we have \[
d(p, g_{i} g_{j}^{-1} p) \le \diam(\mathcal{U}_{l}) \le \epsilon.
\]
Since $p \in S$ is arbitrary, we conclude that $g_{i} g_{j}^{-1}$ is an $\epsilon$-coarse identity. This ends the proof.
\end{proof}

Homeomorphisms of $S$ that are sufficiently close to the identity map in the $C^{0}$-topology are isotopic to the identity map. This fact and Theorem \ref{thm:main} imply the following:

\begin{cor}\label{cor:mainHomotopy}
Let $S$ and $\varphi$ be as in Theorem \ref{thm:main}, and let $x_{0} \in \C*(S)$. Then for each $L>0$, there exists $N$ such that the following holds.

Let $k \in \Z$ and let $g_{1}, g_{2}, \ldots, g_{N+1} \in \Homeo(S)$ be elements of $\Homeo(S)$ that $L$-stabilize $\varphi^{k} x_{0}$ and $\varphi^{k+N} x_{0}$. Then there exists $i \neq j$ such that $g_{i} g_{j}^{-1}$ is isotopic to the identity map on $S$. 
\end{cor}

\section{Application}

Let $X$ be a metric space. We say that two geodesics $\gamma, \eta : [0, L] \rightarrow X$ are \emph{$K$-synchronized} if $d(\gamma(t), \eta(t)) <K$ for each $t$. On a path $\gamma : [0, L] \rightarrow X$, we say that points $\gamma(t_{1}), \ldots, \gamma(t_{N})$ are in order from closest to farthest from $\gamma(0)$, or just \emph{in order} for short, if $0 \le t_{1} \le \ldots \le t_{N}$.

\begin{fact}\label{fact:quadri}
Let $X$ be a geodesic $\delta$-hyperbolic space. Let $K>0$. Let $x, y, x', y' \in X$ be such that $d(x, x') < K$, $d(y, y') < K$ and $d(x, y) >10(K+\delta)$. Then there exist $x_{1}, y_{1} \in [x, y]$ and $x_{2}, y_{2} \in [x', y']$, in order, such that \[
\max\{ d(x, x_{1}), d(y, y_{1}), d(x', x_{2}), d(y', y_{2})\} < K, \quad \max \{d(x_{1}, x_{2}), d(y_{1}, y_{2}) \} < 10\delta.
\]
Furthermore, $x_{1}, x_{2}, y_{1}, y_{2}$ can be chosen such that the geodesics $[x_{1}, y_{1}]$ and $[x_{2}, y_{2}]$ are $10\delta$-synchronized.
\end{fact}

\begin{dfn}\label{dfn:indepdent}
For two sequences $(g_{i})_{i \in \Z}, (h_{i})_{i \in \Z}$ in $\Homeo(S)$, we write $(g_{i})_{i \in \Z} \sim (h_{i})_{i \in \Z}$ if for some (equivalently, any) $x_{0} \in \C*(S)$ there exists $K>0$ such that, for any $i < j$ there exist $k < m$ and $\psi \in \Homeo(S)$ such that  $d_{\C*} (g_{i} x_{0}, \psi h_{k} x_{0}) < K$ and $d_{\C*}(g_{j} x_{0}, \psi h_{m} x_{0}) < K$.

Let $g, h \in \Homeo(S)$ be loxodromic isometries of $\C*(S)$. We write $g \sim h$ if  $(g^{l})_{l \in \Z}\sim (h^{l})_{l \in \Z}$.
\end{dfn}

We will prove the following:

\begin{prop}\label{prop:quantNoQuasi}
Let $S$ be a surface of genus at least 1, let $P \subseteq S$ be a finite set and let $\varphi \in \Homeo(S)$ be a map that restricts to a pseudo-Anosov map on $S \setminus P$. Let $G \le \Homeo(S)$ be a group of homeomorphisms that descends to a non-elementary subgroup of $\Mod(S)$. Then there exists $g_{0} \in G$ that is independent from $\varphi$ (when seen with the action on $\C*(S)$) such that \[
(\ldots, \varphi^{-2}, \varphi^{-1}, id, g_{0}, g_{0}^{2}, \ldots) \not\sim (\varphi^{i})_{i \in \Z} \quad \textrm{and}\quad 
(\ldots, g_{0}^{-2}, g_{0}^{-1}, id, \varphi, \varphi^{2}, \ldots) \not\sim (\varphi^{i})_{i \in \Z}.
\]

If $S$ is a torus and $P$ is a singleton, then the above statement holds for $G \le \Homeo(S)$ that fixes $p$ and descends to a non-elementary subgroup of $\Mod(S \setminus \{p\})$.
\end{prop}

\begin{proof}
In this proof, let $X$ be $\mathcal{C}(S \setminus \{p\})$, the curve complex of the punctured surface, if $S = \mathbb{T}^{2}$ and $P = \{p\}$ is a singleton. Otherwise we let $\mathcal{C} = \mathcal{C}(S)$, the ambient (classical) curve complex. Let $\delta$ be the Gromov hyperbolicity constant for both $\C*(S)$ and $X$. We also fix a basepoint $x_{0} \in \C*(S)$ that does not pass through $P$ and denote its homotopy class by $\mathfrak{x}_{0} \in X$. 
Given $w \in \Homeo(S)$, we write \[
|w|_{\C*} := d_{\C*(S)}(x_{0}, wx_{0}), \quad |w|_X := d_{X} (\mathfrak{x}_{0}, w\mathfrak{x}_{0}).
\]
We also denote the $d_{X}$-translation length of $w$ by $l_{w}$. We then have \[
|w|_X \le |w|_{\C*}, \quad |w^{k}|_X \ge k l_{w} (\forall k \in \Z_{>0}).
\]

Let $G \le \Homeo(S)$ be as in the proposition. Since the image of $G$ in $\Mod(S)$ or $\Mod(S\setminus \{p\})$ is non-elementary, there exist two elements $g, h \in G$ whose mapping classes are independent loxodromics on $X$. Since the $d_{\C*(S)}$-distance is bounded from below by the $d_{X}$-distance (cf. \cite[Lemma 3.8]{bowden2022quasi}), this implies that $g$ and $h$ are independent loxodromics on $\C*(S)$ as well. By replacing $g$ and $h$ with suitable words made of letters $g$ and $h$, we may assume that $\{\varphi, g, h\}$ are independent loxodromics on $\C*(S)$. By powering up $g$, $h$ and $\varphi$ if necessary, we can guarantee the following: 

\begin{obs}\label{obs:ping}
There exist $C> 1000(\delta + 1)$ such that the following hold. \begin{enumerate}
\item $0.99C < l_{\varphi}, l_{g}, l_{h} < C$.
\item Let $w = a_{1} a_{2} \ldots a_{n}$ be a reduced word made of letters $\{g, g^{-1}, h, h^{-1}, \varphi, \varphi^{-1}\}$, i.e., $a_{i} \neq a_{i+1}^{-1}$ for each $i$. Then there exist points $p_{1}, \ldots, p_{n-1} \in [x_{0}, wx_{0}]$, in order from closest to farthest from $x_{0}$, such that $d_{\C*(S)}(p_{i}, a_{1} \cdots a_{i} x_{0}) < 0.01C$.
\end{enumerate}
\end{obs}

We now begin the proof. Let $\mathbf{W} := \{w_{1}, w_{2}, \ldots, \}$ be an infinite collection of words made by $g$ and $h$ such that no powers of $w_{i}$ and $w_{j}$ are equal for distinct $i$ and $j$. $(\ast)$ For example, $w_{i} := g^{i} h$ will do. This property is preserved under taking subsets. Further, for each  $i \neq j$, we have $d_{\C*(S)} (w_{i}^{k} x_{0}, w_{j}^{l} x_{0}) \gtrsim_{i, j} k+l$ for $k, l > 0$.

Let $N$ be as in Corollary \ref{cor:mainHomotopy} for $L = 10C$. We claim:

\begin{claim}\label{claim:singleHanded}
Among any $N+1$ distinct elements in $\mathbf{W}$, there exists $g_{0}$ such that $(\ldots, \varphi^{-2}, \varphi^{-1}, id, g_{0}, g_{0}^{2}, \ldots) \not\sim (\varphi^{l})_{l\in \Z}$. 
\end{claim}

To show this, suppose to the contrary that there exist distinct elements $w_{1}, \ldots, w_{N+1} \in \mathbf{W}$ such that $(\ldots, \varphi^{-2}, \varphi^{-1}, id, w_{i}, w_{i}^{2}, \ldots) \sim (\varphi^{k})_{k \in \Z}$ for $i=1, \ldots, N+1$. That menas, for each $i$ there exists $K_{i}>0$ such that, for each $M>0$ there exists $\psi_{M, i} \in \Homeo(S)$ and $D_{M, i} \in \Z_{>0}$ such that \[
d_{\C*(S)} (\varphi^{-M} x_{0}, \psi_{M, i} x_{0}) < K_{i}, \quad 
d_{\C*(S)} (w_{i}^{M} x_{0}, \psi_{M, i} \varphi^{D_{M, i}} x_{0}) < K_{i}.
\]
Let $W := \sum_{i=1}^{N+1} |w_{i}|_{\C*}$ and let us take an integer $M>100\left(N + \frac{1}{C}\max_{i=1, \ldots, N+1} K_{i} +1\right)$ such that the following implication holds:\begin{equation}\label{eqn:condW}\begin{aligned}
&\left[ (i \neq j \in \{1, \ldots, N\}) \wedge (k, l \in \Z) \wedge  \big( d_{\C*(S)} (x_{0}, w_{i}^{k} x_{0}) \ge 0.5MC\big) \right]\\
& \Rightarrow 
d_{X} \left( w_{i}^{k} \mathfrak{x}_{0}, w_{j}^{l} \mathfrak{x}_{0}\right) \ge 100(C + W).\end{aligned}
\end{equation}
Indeed, large $d_{\C*(S)} (x_{0}, w_{i}^{k} x_{0})$ implies large $k$ (uniform in finitely many $i$), and  $w_{i}^{k}  \mathfrak{x}_{0}$ cannot be $100C$-close to $\{w_{j}^{l} \mathfrak{x}_{0}\}_{l \in \Z}$ for distinct $i$, $j$ and large $k$ due to the independence of $w_{i}$ and $w_{j}$ on $X$.

Let $i \in \{1, \ldots, N+1\}$. Note that $[\varphi^{-M} x_{0}, w_{i}^{M} x_{0}]$ and $[\psi_{M, i} x_{0}, \psi_{M, i} \varphi^{D_{M, i}} x_{0}]$ are geodesics in a $\delta$-hyperbolic space $\C*(S)$ with pairwise $K_{i}$-close endpoints. By Fact \ref{fact:quadri}, there exist a subsegment $\gamma_{i}$ of $[\varphi^{-M} x_{0}, w_{i}^{M} x_{0}]$ and a subsegment $\eta_{i}$ of $[\psi_{M, i} x_{0}, \psi_{M, i} \varphi^{D_{M, i}} x_{0}]$ such that: \begin{enumerate}
\item $\gamma_{i}$ and $\eta_{i}$ are $10\delta$-synchronized;
\item $\gamma_{i}$ and $[\phi^{-M} x_{0}, w_{i}^{M} x_{0}]$ have pairwise $K_{i}$-close endpoints, and  
\item $\eta_{i}$ and $[\psi_{M, i} x_{0}, \psi_{M, i} \varphi^{D_{M, i}} x_{0}$ have pairwise $K_{i}$-close endpoints.
\end{enumerate}
Observation \ref{obs:ping}(2) tells us that there exist $q_{i}, q_{i}' \in  [\varphi^{-M} x_{0}, w_{i}^{M} x_{0}]$ that are $0.01C$-close to $\varphi^{-N} x_{0}, x_{0}$, respectively, where $q_{i}$ comes first. Note that\[\begin{aligned}
d_{\C*(S)}(\varphi^{-M} x_{0}, q_{i}) &\ge
d_{\C*(S)}(\varphi^{-M} x_{0}, \varphi^{-N} x_{0}) - 0.01C \\
&\ge (M - N) \cdot 0.99C - 0.01C > K_{i}.
\end{aligned}
\]
For a similar reason, $q_{i}'$ is $K_{i}$-far from $w_{i}^{M} x_{0}$. That means, $q_{i}, q_{i}'$ are on $\gamma_{i}$. Let $r_{i}, r_{i}' \in \eta_{i}$ be  pairwise $10\delta$-close to $q_{i} ,q_{i}'$, with $r_{i}$ coming first. Then \[\begin{aligned}
d_{\C*(S)} (r_{i}', \textrm{ending point of $\eta_{i}$}) &\ge 
d_{\C*(S)} (r_{i}',  \psi_{M, i} \varphi^{D_{M, i}} x_{0}) -K_{i} \\
&\ge 
d_{\C*(S)} (r_{i}', w_{i}^{2M} x_{0}) - d_{\C*(S)} (w_{i}^{2M} x_{0},  \psi_{M, i} \varphi^{D_{M, i}} x_{0}) - K_{i} \\
&\ge d_{\C*(S)} (q_{i}', w_{i}^{M} x_{0}) - 2K_{i} - 10\delta \\
&\ge 0.99C \cdot M - 2K_{i} - 0.01C \ge 0.95 MC.
\end{aligned}
\]
Hence, we can take a point $r_{i}''$ on $\eta_{i}$ that is right to $r_{i}'$ by distance $0.9MC$. Let $q_{i}''$ be the corresponding point on $\gamma_{i}$ that is $6\delta$-close to $r_{i}''$.

By Observation \ref{obs:ping}(2), $r_{i}, r_{i}', r_{i}''$ are $1.1C$-close to $\psi_{M, i}\varphi^{t(i)} x_{0}$, $\psi_{M, i}\varphi^{t'(i)} x_{0}$ and $\psi_{M, i}\varphi^{t''(i)} x_{0}$, respectively, for some $t(i) \le t'(i) \le t''(i)$. This implies \[\begin{aligned}
|\varphi^{t'(i) - t(i)}|_{\C*} &=d_{\C*(S)}( \psi_{M, i}\varphi^{t'(i)} x_{0}, \psi_{M, i} \varphi^{t(i)} x_{0}) \\
&\le d_{\C*(S)} (r_{i}, r_{i}') + 2.2C \le d_{\C*(S)} (\varphi^{-N}x_{0}, x_{0}) + 2.3C.
\end{aligned}
\]
Similarly, $|\varphi^{t'(i) - t(i)}|_{\C*}  \ge|\varphi^{N}|_{\C*} - 2.3C$ holds. Meanwhile, Observation \ref{obs:ping}(2) applied to the word $\varphi^{k} = \varphi \cdot \varphi \cdots \varphi$ implies the following: for each $0<l< k$, $|\varphi^{k}|_{\C*}$ and $|\varphi^{l}|_{\C*}$ differ by at least $|\varphi^{k-l}|_{\C*} - 0.02C \ge 0.9C|k-l|$. Combining these together, we observe that $t'(i) -t(i)$ and $N$ differ by at most 2. This in turn implies that $r_{i}'$ and $\psi_{M, i} \varphi^{t(i) + N} x_{0}$ is $3.1C$-close. Similarly, $t''(i) - t'(i)$ is almost constant in $i$: there exists $T \in \Z>0$ such that $|t''(i) - t'(i) - T| \le 2$ regardless of $i$. We thus observe that $\psi_{M, i} \varphi^{t(i) + N + T}x_{0}$ and $r_{i}''$ is $5.1C$-close. 

Finally, recall that $q_{i}''$ is a point on $[\varphi^{-M} x_{0}, w_{i}^{M} x_{0}]$ that appears later than $q_{i}'$. Hence, it is $(0.01C + |w_{i}|)$-close to $w_{i}^{k(i)} x_{0}$ for some $k(i) > 0$, by Observation \ref{obs:ping}(2). 

In summary, for $\phi_{i} := \psi_{M, i} \varphi^{t(i)+N}$ we have
 \[
d_{\C*(S)} (\varphi^{-N}x_{0}, \phi_{i} \varphi^{-N} x_{0}) < 1.2C, d_{\C*(S)} (x_{0}, \phi_{i} x_{0}) < 3.2C, d_{\C*(S)} (w_{i}^{k(i)} x_{0}, \phi_{i} \varphi^{T} x_{0}) < 5.2C + W.
\]
We now have $N+1$ homeomorphisms $\phi_{1}, \ldots, \phi_{N+1}$ that $L$-stabilize $\varphi^{-N} x_{0}$ and $\varphi^{N} x_{0}$. By Corollary \ref{cor:mainHomotopy}, there exist $i \neq j$ such that $\phi_{i} \phi_{j}^{-1}$ is isotopic to the identity map on $S$. In particular, the following curves: \[
\phi_{j} \varphi^{T} x_{0}, \quad \phi_{i} \varphi^{T} x_{0} = \phi_{i} \phi_{j}^{-1}\cdot \phi_{j} \varphi^{T} x_{0}
\]
are homotopic \textbf{on $S$}. If $S$ is a torus and $P$ is a singleton, then these curves are homotopic even on $S \setminus p$, a once-punctured torus. In conclusion, they correspond to  the same points on $X$.

This implies that \[\begin{aligned}
d_{X}(w_{i}^{k(i)} \mathfrak{x}_{0}, w_{j}^{k(j)} \mathfrak{x}_{0}) & \le 
d_{X}(w_{i}^{k(i)} \mathfrak{x}_{0}, \phi_{i} \varphi^{T} \mathfrak{x}_{0})   + 
d_{X}(\phi_{i} \varphi^{T} \mathfrak{x}_{0} ,w_{j}^{k(j)} \mathfrak{x}_{0})   \\
&= 
d_{X}(w_{i}^{k(i)} \mathfrak{x}_{0}, \phi_{i} \varphi^{T} \mathfrak{x}_{0})   + 
d_{X}(\phi_{j} \varphi^{T} \mathfrak{x}_{0} ,w_{j}^{k(j)} \mathfrak{x}_{0})  \\
&\le d_{\C*}(w_{i}^{k(i)} x_{0}, \phi_{i} \varphi^{T} x_{0})   + 
d_{\C*}(\phi_{j} \varphi^{T} x_{0} ,w_{j}^{k(j)}x_{0})  \le 11C + 2W.
\end{aligned}
\]

Meanwhile, recall that \[\begin{aligned}
d_{\C*(S)} (w_{i}^{k(i)}x_{0}, x_0) &\ge 
d_{\C*(S)}(r_{i}'', r_{i}') - d_{\C*(S)} (w_{i}^{k(i)}x_{0}, \phi_{i} \varphi^{T} x_0) - d_{\C*(S)} (\psi_{M,i} \varphi^{t(i) + N + T} x_0, r_{i}'') \\
&- d_{\C*(S)} (r_{i}', \psi_{M, i} \varphi^{t(i) + N}x_{0}) \ge 0.6MC.
\end{aligned}
\]
This contradicts our requirement for $M$ in Display \ref{eqn:condW}. Claim \ref{claim:singleHanded} now follows.

Using Claim \ref{claim:singleHanded}, we can construct an infinite subset $\mathbf{W}' \subseteq \mathbf{W}$ such that $(\ldots, \varphi^{-2}, \varphi^{-1}, id, g_{0}, g_{0}^{2}, \ldots) \not\sim (\varphi^{l})_{l \in \Z}$  for all $g_{0} \in \mathbf{W}'$. Now using a symmetric argument, we can find $g_{0}$ such that $(\ldots, g_{0}^{-2}, g_{0}^{-1}, id, \varphi, \varphi^{2}, \ldots) \not\sim (\varphi^{l})_{l \in \Z}$ as well. This ends the proof.
\end{proof}

\subsection{Construction of quasi-morphisms}
We now proceed to prove Theorem \ref{thm:inf}. Let $S$ be a surface of genus at least 1, let $U$ be an open disc on $S$, and let $p \in U$. There exists a pseudo-Anosov map $\varphi$ on $S \setminus \{p\}$. Note that $\varphi$ never belongs to $Fix_{U}$, as it mixes up $U \setminus \{p\}$ and $S \setminus U$. Nonetheless, $Fix_{U}$ does descend to the entire $\Mod(S \setminus \{p\})$ after quotienting by homotopy, so Proposition \ref{prop:quantNoQuasi} guarantees an element $g \in Fix_{U}$ such that \begin{equation}\label{eqn:nonInd}
(\ldots, \varphi^{-2}, \varphi^{-1}, id, g, g^{2}, \ldots) \not\sim (\varphi^{i})_{i \in \Z} \quad \textrm{and}\quad 
(\ldots, g^{-2}, g^{-1}, id, \varphi, \varphi^{2}, \ldots) \not\sim (\varphi^{i})_{i \in \Z}.
\end{equation}
Note that we can take such a $g$ in $\Homeo_{0}(S)$ when $S = \mathbb{T}^{2}$. In general we will have $g \neq id$ in $\Mod(S)$.

If $(g_{i})_{i \in \Z} \not\sim (h_{i})_{i \in \Z}$, then $(g_{i}')_{i \in \Z} \not\sim (h_{i}')_{i \in \Z}$ for any subsequence $(g_{i}')_{i \in \Z}$ of $(g_{i})_{i \in \Z}$ and $(h_{i}')_{i}$ of $(h_{i})_{i \in \Z}$. Hence the condition in Display \ref{eqn:nonInd} is preserved if we replace $g$ and $\varphi$ with their powers, respectively. Given this, it is not hard to guarantee Observation \ref{obs:ping}. Namely, up to replacing $g$ and $\varphi$ with their suitable powers, the following hold for some $C > 1000(\delta+1)$: \begin{enumerate}
\item $0.99C \le l_{g}, l_{\varphi} \le C$;
\item Let $w = a_{1} a_{2} \ldots a_{n}$ be a reduced word made of letters $\{g, g^{-1}, \varphi, \varphi^{-1}\}$, i.e., $a_{i} \neq a_{i+1}^{-1}$ for each $i$. Then there exist points $p_{1}, \ldots, p_{n-1} \in [x_{0}, wx_{0}]$, in order from closest to farthest from $x_{0}$, such that $d_{\C*(S)}(p_{i}, a_{1} \cdots a_{i} x_{0}) < 0.01C$.
\end{enumerate}

We will stick to these choices of $U$, $\varphi$, $g$. Recall the notation for each $w \in \Homeo(S)$: $l_{w}$ denotes the $d_{\C*(S)}$-translation lengths of $w$ and $|w|_{\C*}$ denotes  $d_{\C*(S)}(x_{0}, wx_{0})$. When we write $a \le_{K} b$ $(a \ge_{K} b$, resp.) it means $a \le b + K$ $(a \ge b - K$, resp.).

\begin{fact}\label{fact:separation}
\begin{enumerate}
\item For any $i_{2}, \ldots, i_{2N-1} \in \Z \setminus \{0\}$ and $i_{1}, i_{2N} \in \Z$, we have \[
\big|g^{i_{1}} \varphi^{i_{2}} \cdots  g^{i_{2N-1}} \varphi^{i_{2N}}\big|_{\C*} =_{0.1(N+1) C} \sum_{k=1}^{N} \Big( \big|i_{2k-1}\big| \cdot l_{g} + \big|i_{2k}\big| \cdot l_{\varphi}\Big).
\]
In particular, $d_{\C*(S)} (x_{0}, g^{a} \varphi^{b} x_{0}) =_{0.1 C} al_{g} + bl_{\varphi}$. 
\item There exists $K_{sep}>0$ such that the following holds.

Let $w$ and $v$ be reduced words made of letters $\{ g, g^{-1}, \varphi, \varphi^{-1}\}$ and let $\phi \in \Homeo(S)$ be such that $d(x_{0}, \phi x_{0}) < 10C$ and  $d(wx_{0}, \phi vx_{0}) < 10C$. Let \[
w = w_{f} w_{0} w_{b}, \quad v = v_{f} v_{0} v_{1} v_{b}
\]
be subword decompositions of $w$ and $v$, where $w_{0}, v_{0}$ are positive powers of $\varphi$ and $v_{1}$ is a positive power of $g$. Then the following inequalities cannot hold simultaneously: \begin{equation}\label{eqn:contra4Ineq}
|w_{f}|_{\C*} \le |v_{f}|_{\C*}+100C, \quad |w_{b}|_{\C*} \le |v_{b}|_{\C*} + 100C,\quad  |v_{0}|_{\C*} \ge K_{sep},\quad  |v_{1}|_{\C*} \ge K_{sep}.
\end{equation}
\end{enumerate}
\end{fact}

The first item is a direct consequence of the triangle inequality and Observation \ref{obs:ping}(2). In the second item, if the 4 inequalities hold at the same time, then there exists $\phi_{k, l} \in \Homeo(S)$ and $i(k, l)< j(k, l)$ for each $k, l \le 10(K_{sep} - 1000C)/C$ such that \[
d_{\C*(S)} (\varphi^{-k} x_{0}, \phi_{k, l} \phi^{i(k, l)} x_{0}) \le 10C, \,\,
d_{\C*(S)} (g_{l} x_{0}, \phi_{k, l} \phi^{j(k, l)} x_{0}) \le 10C.
\]
If this is true for arbitrary $K_{sep}$, then $(\ldots, \varphi^{-2}, \varphi^{-1}, id, g, g^{2}, \ldots) \sim (\varphi^{i})_{i \in \Z}$, contradicting our assumption. Hence, the 4 inequalities do not hold simultaneously for some $K_{sep}$.

More generally, the following holds for the same reason:
\begin{fact} \label{fact:separationStrong}
The following holds for some $K_{sep}>0$. 

Let $\epsilon \in \{+1, -1\}$. Let $w$ and $v$ be reduced words made of letters $\{ g, g^{-1}, \varphi, \varphi^{-1}\}$ and let $\phi \in \Homeo(S)$ be such that $d(x_{0}, \phi x_{0}) < 10C$ and  $d(wx_{0}, \phi vx_{0}) < 10C$. Let \[
w = w_{f} w_{0} w_{b}, \quad v = v_{f} v_{0} v_{1} v_{b}
\]
be subword decompositions of $w$ and $v$. Suppose that $w_{0}$ and one of $\{v_{0}, v_{1}\}$ are positive powers of $\varphi^{\epsilon}$, and the other one of $\{v_{0}, v_{1}\}$ is a positive power of $g^{\epsilon}$. Then the following inequalities cannot hold simultaneously: \[
|w_{f}|_{\C*} \le |v_{f}|_{\C*}- 100C, \quad |w_{b}|_{\C*} \le |v_{b}|_{\C*} - 100C,\quad  |v_{0}|_{\C*} \ge K_{sep},\quad  |v_{1}|_{\C*} \ge K_{sep}.
\]
\end{fact}

We will now take elements of the form \[
\psi_{a, b, c, d, e, f} := g^{a} \cdot \varphi^{f} g^{b} \varphi^{-f} \cdot g^{c} \cdot \varphi^{e} g^{-d} \varphi^{-e}.
\]
Our first claim is: \begin{claim}\label{claim:psiFixes}
For each $a, b, c, d, e, f \in \Z$, $\psi_{a, b, c, d, e, f}$ fixes a neighborhood of $p$. When $a+b+c = d$ in addition, $\psi_{a, b, c, d, e, f}$ belongs to $\Homeo_{0}(S)$.
\end{claim}

\begin{proof}
The former claim follows because  $g^{a}$ and $g^{c}$ fix $U \ni p$, $\varphi^{f} g^{b} \varphi^{-f}$ fixes $\varphi^{f} U \ni p$, and $\varphi^{e} g^{-d} \varphi^{-e}$ fixes $\varphi^{e} U\ni p$. The latter claim is immediate.
\end{proof}

Next, we observe: \begin{claim}\label{claim:psiNonQuasiInv}
For $1 \ll a \ll b \ll c \le d  \ll e \ll f$ , $\psi_{a, b, c, d, e, f}$ is not quasi-invertible, i.e., $\psi_{a, b, c, d, e, f} \not\sim \psi_{a, b, c, d, e, f}^{-1}$.
\end{claim}

\begin{proof}
For convenience, let us write $\Xi := a+b+c+d+2e+2f$ and  $\psi := \psi_{a, b, c, d, e, f}$. Suppose to the contrary that $\psi \sim \psi^{-1}$. More explicitly, for some $K>0$, suppose for each $T \in \Z_{>0}$ there exist $\phi_{T} \in \Homeo(S)$ and $i<j$ such that \begin{equation}\label{eqn:contraT}
d_{\C*(S)} ( \psi^{-T} x_{0}, \phi_{T} \psi^{j} x_{0}) < K, \quad 
d_{\C*(S)} ( \psi^{T} x_{0}, \phi_{T} \psi^{i} x_{0}) < K.
\end{equation}
We pick $T\ge 100$ large enough such that \[
0.5T \cdot \Xi C \ge 10(K + C).
\]

For this $T$, we denote $[ \psi^{-T} x_{0}, \psi^{T} x_{0}]$ by $\gamma$ and $[\varphi_{T} \psi^{j} x_{0}, \phi_{T} \psi^{i} x_{0}]$ by $\eta$. They have pairwise $K$-close endpoints. Fact \ref{fact:quadri} guarantees points $x, y \in \gamma$, in order, and points $x', y' \in \eta$, in order, such that $[x, y]$ and $[x', y']$ are $10\delta$-synchronized and\[
d(x, y) > d(\psi^{-T} x_{0}, \psi^{T} x_{0}) - 2K \ge 1.98T\cdot \Xi C - 2K \ge 10\Xi C + 10C.
\]

Let us define subwords of $\psi = g^{a} \varphi^{f} g^{b} \varphi^{-f} g^{c} \varphi^{e} g^{-d} \varphi^{-e}$: \[\begin{aligned}
w_{1} := g, \ldots, w_{a} := g^{a}, w_{a+1} := g^{a} \varphi, \ldots, w_{a+f} := g^{a} \varphi^{f},\\
 \ldots, w_{a+2f + b} := g^{a} \varphi^{f} g^{b} \varphi^{-f}, \ldots, w_{\Xi} = w_{a+b+c+d+2e+2f} := \psi.
\end{aligned}
\]
We then define $w_{k \Xi + i} := \psi^{k} w_{i}$ for each $k \in \Z$ and $i \in \{1, \ldots, \Xi\}$.  Fact \ref{fact:separation}(1) tells us that \begin{equation}\label{eqn:constantWWP}
|w_{k}^{-1} w_{k + \Xi}|_{\C*} =_{C} (a+b+c+d)l_{g} + (2e+2f)l_{\varphi} \quad (\forall k \in \Z).
\end{equation}For example, for $0 \le k \le a$ we have \[\begin{aligned}
|w_{k}^{-1} w_{k + \Xi}|_{\C*} &= |g^{a-k} \varphi^{f} g^{b} \varphi^{-f} g^{c} \varphi^{e} g^{-d} \varphi^{-e} g^{k}|_{\C*} \\
&=_{0.1(5+1) \cdot C} ((a-k)+b+c+d + k)l_{g} + (2e+2f+0)l_{\varphi} \\
&=  (a+b+c+d)l_{g} + (2e+2f)l_{\varphi}.
\end{aligned}
\]

By Observation \ref{obs:ping}, there are points $p_{0}, \ldots, p_{2T\Xi} \in \gamma$, in order, such that $d_{\C*(S)} (w_{l} x_{0}, p_{l}) < 0.01C$ for each $l \in \{0, \ldots, 2T\Xi\}$. Let $A$ be the smallest integer such that $p_{A \cdot \Xi}$ comes later than $x \in \gamma$. In this case, $x$ lies in between $p_{(A-1)\Xi}$ and $p_{A \Xi}$, which are spaced by at most $1.1C \cdot \Xi + 0.02C$. Also, $p_{(A+1) \Xi}$ is $(1.1C \cdot \Xi + 0.02C)$-close to $p_{A \Xi}$. Hence, \[
d_{\C*(S)} (x, p_{(A+1)\Xi}) \le 2.2C \Xi + 0.04C \le d_{\C*(S)}(x, y).
\]
Since both $p_{A\Xi}$ and $p_{(A+1)\Xi}$ belongs to $[x, y]$, they are $10\delta$-close to $q, q' \in [x', y']$, respectively, with $q$ appearing earlier. Note that $d_{\C*(S)}(q, q')$ is $(20\delta + 0.04C)$-close to $d_{\C*(S)} (w_{A \cdot \Xi} x_{0}, w_{(A+1) \cdot \Xi} x_{0}) = |\psi|_{\C*}$. ($\ast$)

Meanwhile, Observation \ref{obs:ping} tells us that $\eta = \phi_{T} \psi^{i} \cdot [\psi^{j-i} x_{0}, x_{0}]$ contains a $1.1C$-dense sequence (in order) that $0.01C$-fellow travels with \[\begin{aligned}
& ( \phi_{T} \psi^{i} w_{(j-i) \Xi - l} x_{0})_{l = 0, \ldots, (j-i)\Xi} \\
 &= \big(\phi_{T} \psi^{j} x_{0} = \phi_{T} \psi^{i} w_{(j-i)} x_{0},\,\, \phi_{T} \psi^{i} w_{(j-i)-1} x_{0},\,\, \ldots,\,\, \phi_{T} \psi^{i} w_{1} x_{0},\,\, \phi_{T} \psi^{i} w_{0} x_{0} = \phi_{T} \psi^{i} x_{0} \big).
 \end{aligned}
 \]Hence, there exist an integer $B \in \{0, \ldots, (j-i)\Xi\}$ such that $\phi_{T} \psi^{i} \cdot w_{B} x_{0}$ is $C$-close to $q$. Let $q'' \in \eta$ be the point that is $0.01C$-close to $\phi_{T} \psi^{i} \cdot w_{B - \Xi} x_{0}$, which appears later than $q$. In this case, $d_{\C*(S)}(q, q'')$ is $(C + 0.01C)$-close to $d_{\C*(S)} (w_{B} x_{0}, w_{B - \Xi} x_{0}) =_{C} |\psi|_{\C*}$ (by Display \ref{eqn:constantWWP}). Finally, both $q'$ and $q''$ are to the right of $q$. Combining these facts with $(\ast$), we conclude that $d_{\C*(S)}(q', q'') \le 3C$. In summary, we have  \[
 d_{\C*(S)} ( \psi^{A} x_{0}, \phi_{T} \psi^{i} w_{B} x_{0}) < 10C, \quad  d_{\C*(S)} ( \psi^{A+1} x_{0}, \phi_{T} \psi^{i} w_{B - \Xi} x_{0}) < 10C.
 \]
 In other words, if we let $v := w_{B}^{-1} w_{B - \Xi}$ and $\phi := \psi_{A}^{-1} \phi_{T} \psi^{i} w_{B} \in \Homeo(S)$, then \begin{equation}\label{eqn:contraP}
 d_{\C*(S)} (x_{0}, \phi x_{0}) < 10C, \quad  d_{\C*(S)} ( \psi x_{0}, \phi vx_{0}) < 10C.
 \end{equation}
 
It remains to deduce contradiction from Inequality \ref{eqn:contraP}, given $K_{sep} \ll a \ll \ldots \ll f$. Let $R$ be the remainder of $B$ when divided by $\Xi$. See Figure \ref{fig:phiPowergPower}.

\begin{figure}
\centerline{\begin{tikzpicture}
\def\a{0.2}
\def\b{0.32}
\def\c{0.75}
\def\d{1.1}
\def\e{2}
\def\f{6.5}d
\def\s{0.2}
\def\vs{0.55}
\def\vss{0.35}
\def\col{0.13}

\clip (0,-3.5) rectangle (\a+\b+\c+\d+2*\e+2*\f, 0.3);

\draw[dashed] (\a, -3.3) rectangle (\a + \f, 0.2);

\draw[dashed](\a+\f+\b, -3.3) rectangle (\a + 2*\f+\b, 0.2);

\draw[dashed](\a+2*\f+\b + \c, -3.3) rectangle (\a +2*\f + \b + \c + \e, 0.2);

\draw[ dashed](\a+2*\f+\b + \c + \e + \d, -3.3) rectangle (\a +2*\f + \b + \c + \d + 2*\e, 0.2);

\draw[->, red, thick] (0, 0) -- (\a, 0);
\draw[->, blue, thick] (\a, 0) -- (\a+\f, 0);
\draw[->, red, thick] (\a + \f, 0) -- (\a + \f + \b, 0);
\draw[<-, blue, thick] (\a+\f+\b, 0) -- (\a + 2*\f + \b, 0);
\draw[->, red, thick] (\a + 2*\f + \b, 0)-- (\a + 2*\f + \b + \c, 0);
\draw[->, blue, thick]  (\a + 2*\f + \b + \c, 0)-- (\a + 2*\f + \b + \c + \e ,0);
\draw[<-, red, thick]   (\a + 2*\f + \b + \c + \e ,0)-- (\a + 2*\f + \b + \c + \e+\d ,0);
\draw[<-, blue, thick]   (\a + 2*\f + \b + \c + \e+\d ,0)-- (\a + 2*\f + \b + \c +2*\e+\d ,0);

\begin{scope}[xscale=-1, shift={(-\a-\b-\c-\d-2*\e-2*\f, -\vs)}]

\fill[opacity=0.13] (\a + \f - \s, -\col) rectangle (\a + \f + \s, \col);

\fill[opacity=0.13] (\a + 2*\f + \b + \c + \e+\d - \s, -\col) rectangle (\a + 2*\f + \b + \c + \e+\d + \s, \col);

\draw[->, red, thick] (0, 0) -- (\a, 0);
\draw[->, blue, thick] (\a, 0) -- (\a+\f, 0);
\draw[->, red, thick] (\a + \f, 0) -- (\a + \f + \b, 0);
\draw[<-, blue, thick] (\a+\f+\b, 0) -- (\a + 2*\f + \b, 0);
\draw[->, red, thick] (\a + 2*\f + \b, 0)-- (\a + 2*\f + \b + \c, 0);
\draw[->, blue, thick]  (\a + 2*\f + \b + \c, 0)-- (\a + 2*\f + \b + \c + \e ,0);
\draw[<-, red, thick]   (\a + 2*\f + \b + \c + \e ,0)-- (\a + 2*\f + \b + \c + \e+\d ,0);
\draw[<-, blue, thick]   (\a + 2*\f + \b + \c + \e+\d ,0)-- (\a + 2*\f + \b + \c +2*\e+\d ,0);

\end{scope}

\begin{scope}[xscale=-1]

\begin{scope}[shift={(-\a-\b-\c-\d-2*\e-2*\f - \a, -\vs - \vss)}]

\fill[opacity=0.13] (\a + 2*\f + \b + \c + \e+\d - \s, -\col) rectangle (\a + 2*\f + \b + \c + \e+\d + \s, \col);

\draw[->, red, thick] (0, 0) -- (\a, 0);
\draw[->, blue, thick] (\a, 0) -- (\a+\f, 0);
\draw[->, red, thick] (\a + \f, 0) -- (\a + \f + \b, 0);
\draw[<-, blue, thick] (\a+\f+\b, 0) -- (\a + 2*\f + \b, 0);
\draw[->, red, thick] (\a + 2*\f + \b, 0)-- (\a + 2*\f + \b + \c, 0);
\draw[->, blue, thick]  (\a + 2*\f + \b + \c, 0)-- (\a + 2*\f + \b + \c + \e ,0);
\draw[<-, red, thick]   (\a + 2*\f + \b + \c + \e ,0)-- (\a + 2*\f + \b + \c + \e+\d ,0);
\draw[<-, blue, thick]   (\a + 2*\f + \b + \c + \e+\d ,0)-- (\a + 2*\f + \b + \c +2*\e+\d ,0);

\end{scope}

\begin{scope}[shift={(- \a, -\vs - \vss)}]
\draw[->, red, thick] (0, 0) -- (\a, 0);
\draw[->, blue, thick] (\a, 0) -- (\a+\f, 0);
\draw[->, red, thick] (\a + \f, 0) -- (\a + \f + \b, 0);
\draw[<-, blue, thick] (\a+\f+\b, 0) -- (\a + 2*\f + \b, 0);
\draw[->, red, thick] (\a + 2*\f + \b, 0)-- (\a + 2*\f + \b + \c, 0);
\draw[->, blue, thick]  (\a + 2*\f + \b + \c, 0)-- (\a + 2*\f + \b + \c + \e ,0);
\draw[<-, red, thick]   (\a + 2*\f + \b + \c + \e ,0)-- (\a + 2*\f + \b + \c + \e+\d ,0);
\draw[<-, blue, thick]   (\a + 2*\f + \b + \c + \e+\d ,0)-- (\a + 2*\f + \b + \c +2*\e+\d ,0);
\end{scope}
\end{scope}


\begin{scope}[xscale=-1]
\begin{scope}[shift={(-\a-\b-\c-\d-2*\e-2*\f-\f - \a +\e + \d, -\vs - \vss*2)}]

\fill[opacity=0.13] (\a + 2*\f + \b + \c + \e+\d - \s, -\col) rectangle (\a + 2*\f + \b + \c + \e+\d + \s, \col);

\fill[opacity=0.13] (\a + 2*\f + \b + \c  - \s, -\col) rectangle (\a + 2*\f + \b + \c + \s, \col);

\draw[->, red, thick] (0, 0) -- (\a, 0);
\draw[->, blue, thick] (\a, 0) -- (\a+\f, 0);
\draw[->, red, thick] (\a + \f, 0) -- (\a + \f + \b, 0);
\draw[<-, blue, thick] (\a+\f+\b, 0) -- (\a + 2*\f + \b, 0);
\draw[->, red, thick] (\a + 2*\f + \b, 0)-- (\a + 2*\f + \b + \c, 0);
\draw[->, blue, thick]  (\a + 2*\f + \b + \c, 0)-- (\a + 2*\f + \b + \c + \e ,0);
\draw[<-, red, thick]   (\a + 2*\f + \b + \c + \e ,0)-- (\a + 2*\f + \b + \c + \e+\d ,0);
\draw[<-, blue, thick]   (\a + 2*\f + \b + \c + \e+\d ,0)-- (\a + 2*\f + \b + \c +2*\e+\d ,0);

\end{scope}

\begin{scope}[shift={(- \f -\a+ \e + \d, -\vs - \vss*2)}]
\draw[->, red, thick] (0, 0) -- (\a, 0);
\draw[->, blue, thick] (\a, 0) -- (\a+\f, 0);
\draw[->, red, thick] (\a + \f, 0) -- (\a + \f + \b, 0);
\draw[<-, blue, thick] (\a+\f+\b, 0) -- (\a + 2*\f + \b, 0);
\draw[->, red, thick] (\a + 2*\f + \b, 0)-- (\a + 2*\f + \b + \c, 0);
\draw[->, blue, thick]  (\a + 2*\f + \b + \c, 0)-- (\a + 2*\f + \b + \c + \e ,0);
\draw[<-, red, thick]   (\a + 2*\f + \b + \c + \e ,0)-- (\a + 2*\f + \b + \c + \e+\d ,0);
\draw[<-, blue, thick]   (\a + 2*\f + \b + \c + \e+\d ,0)-- (\a + 2*\f + \b + \c +2*\e+\d ,0);
\end{scope}

\end{scope}


\begin{scope}[xscale=-1]
\begin{scope}[shift={(-\a-\b-\c-\d-2*\e-2*\f -\f - \a, -\vs - \vss*3)}]

\fill[opacity=0.13] (\a + 2*\f + \b + \c  - \s, -\col) rectangle (\a + 2*\f + \b + \c + \s, \col);

\draw[->, red, thick] (0, 0) -- (\a, 0);
\draw[->, blue, thick] (\a, 0) -- (\a+\f, 0);
\draw[->, red, thick] (\a + \f, 0) -- (\a + \f + \b, 0);
\draw[<-, blue, thick] (\a+\f+\b, 0) -- (\a + 2*\f + \b, 0);
\draw[->, red, thick] (\a + 2*\f + \b, 0)-- (\a + 2*\f + \b + \c, 0);
\draw[->, blue, thick]  (\a + 2*\f + \b + \c, 0)-- (\a + 2*\f + \b + \c + \e ,0);
\draw[<-, red, thick]   (\a + 2*\f + \b + \c + \e ,0)-- (\a + 2*\f + \b + \c + \e+\d ,0);
\draw[<-, blue, thick]   (\a + 2*\f + \b + \c + \e+\d ,0)-- (\a + 2*\f + \b + \c +2*\e+\d ,0);

\end{scope}

\begin{scope}[shift={(-\f - \a, -\vs - \vss*3)}]
\draw[->, red, thick] (0, 0) -- (\a, 0);
\draw[->, blue, thick] (\a, 0) -- (\a+\f, 0);
\draw[->, red, thick] (\a + \f, 0) -- (\a + \f + \b, 0);
\draw[<-, blue, thick] (\a+\f+\b, 0) -- (\a + 2*\f + \b, 0);
\draw[->, red, thick] (\a + 2*\f + \b, 0)-- (\a + 2*\f + \b + \c, 0);
\draw[->, blue, thick]  (\a + 2*\f + \b + \c, 0)-- (\a + 2*\f + \b + \c + \e ,0);
\draw[<-, red, thick]   (\a + 2*\f + \b + \c + \e ,0)-- (\a + 2*\f + \b + \c + \e+\d ,0);
\draw[<-, blue, thick]   (\a + 2*\f + \b + \c + \e+\d ,0)-- (\a + 2*\f + \b + \c +2*\e+\d ,0);
\end{scope}

\end{scope}


\begin{scope}[xscale=-1]
\begin{scope}[shift={(-\a-\b-\c-\d-2*\e-2*\f - 2*\f - \a - \b + 2*\e + \c + \d  , -\vs - \vss*4)}]

\fill[opacity=0.13] (\a + 2*\f + \b + \c  - \s, -\col) rectangle (\a + 2*\f + \b + \c + \s, \col);

\draw[->, red, thick] (0, 0) -- (\a, 0);
\draw[->, blue, thick] (\a, 0) -- (\a+\f, 0);
\draw[->, red, thick] (\a + \f, 0) -- (\a + \f + \b, 0);
\draw[<-, blue, thick] (\a+\f+\b, 0) -- (\a + 2*\f + \b, 0);
\draw[->, red, thick] (\a + 2*\f + \b, 0)-- (\a + 2*\f + \b + \c, 0);
\draw[->, blue, thick]  (\a + 2*\f + \b + \c, 0)-- (\a + 2*\f + \b + \c + \e ,0);
\draw[<-, red, thick]   (\a + 2*\f + \b + \c + \e ,0)-- (\a + 2*\f + \b + \c + \e+\d ,0);
\draw[<-, blue, thick]   (\a + 2*\f + \b + \c + \e+\d ,0)-- (\a + 2*\f + \b + \c +2*\e+\d ,0);

\end{scope}

\begin{scope}[shift={(-2*\f  - \a - \b + 2*\e + \c + \d , -\vs - \vss*4)}]

\fill[opacity=0.13] (\a  - \s, -\col) rectangle (\a + \s, \col);

\draw[->, red, thick] (0, 0) -- (\a, 0);
\draw[->, blue, thick] (\a, 0) -- (\a+\f, 0);
\draw[->, red, thick] (\a + \f, 0) -- (\a + \f + \b, 0);
\draw[<-, blue, thick] (\a+\f+\b, 0) -- (\a + 2*\f + \b, 0);
\draw[->, red, thick] (\a + 2*\f + \b, 0)-- (\a + 2*\f + \b + \c, 0);
\draw[->, blue, thick]  (\a + 2*\f + \b + \c, 0)-- (\a + 2*\f + \b + \c + \e ,0);
\draw[<-, red, thick]   (\a + 2*\f + \b + \c + \e ,0)-- (\a + 2*\f + \b + \c + \e+\d ,0);
\draw[<-, blue, thick]   (\a + 2*\f + \b + \c + \e+\d ,0)-- (\a + 2*\f + \b + \c +2*\e+\d ,0);

\end{scope}

\end{scope}


\begin{scope}[xscale=-1]
\begin{scope}[shift={(-\a-\b-\c-\d-2*\e-2*\f -2*\f - \a - \b, -\vs - \vss*5)}]
\draw[->, red, thick] (0, 0) -- (\a, 0);
\draw[->, blue, thick] (\a, 0) -- (\a+\f, 0);
\draw[->, red, thick] (\a + \f, 0) -- (\a + \f + \b, 0);
\draw[<-, blue, thick] (\a+\f+\b, 0) -- (\a + 2*\f + \b, 0);
\draw[->, red, thick] (\a + 2*\f + \b, 0)-- (\a + 2*\f + \b + \c, 0);
\draw[->, blue, thick]  (\a + 2*\f + \b + \c, 0)-- (\a + 2*\f + \b + \c + \e ,0);
\draw[<-, red, thick]   (\a + 2*\f + \b + \c + \e ,0)-- (\a + 2*\f + \b + \c + \e+\d ,0);
\draw[<-, blue, thick]   (\a + 2*\f + \b + \c + \e+\d ,0)-- (\a + 2*\f + \b + \c +2*\e+\d ,0);

\end{scope}

\begin{scope}[shift={(-2*\f - \a - \b, -\vs - \vss*5)}]

\fill[opacity=0.13] (\a  - \s, -\col) rectangle (\a + \s, \col);

\fill[opacity=0.13] (-\e  - \s, -\col) rectangle (-\e + \s, \col);

\draw[->, red, thick] (0, 0) -- (\a, 0);
\draw[->, blue, thick] (\a, 0) -- (\a+\f, 0);
\draw[->, red, thick] (\a + \f, 0) -- (\a + \f + \b, 0);
\draw[<-, blue, thick] (\a+\f+\b, 0) -- (\a + 2*\f + \b, 0);
\draw[->, red, thick] (\a + 2*\f + \b, 0)-- (\a + 2*\f + \b + \c, 0);
\draw[->, blue, thick]  (\a + 2*\f + \b + \c, 0)-- (\a + 2*\f + \b + \c + \e ,0);
\draw[<-, red, thick]   (\a + 2*\f + \b + \c + \e ,0)-- (\a + 2*\f + \b + \c + \e+\d ,0);
\draw[<-, blue, thick]   (\a + 2*\f + \b + \c + \e+\d ,0)-- (\a + 2*\f + \b + \c +2*\e+\d ,0);

\end{scope}

\end{scope}


\begin{scope}[xscale=-1]
\begin{scope}[shift={(-\a-\b-\c-\d-2*\e-2*\f -2*\f - 2*\a - 2*\b, -\vs - \vss*6)}]
\draw[->, red, thick] (0, 0) -- (\a, 0);
\draw[->, blue, thick] (\a, 0) -- (\a+\f, 0);
\draw[->, red, thick] (\a + \f, 0) -- (\a + \f + \b, 0);
\draw[<-, blue, thick] (\a+\f+\b, 0) -- (\a + 2*\f + \b, 0);
\draw[->, red, thick] (\a + 2*\f + \b, 0)-- (\a + 2*\f + \b + \c, 0);
\draw[->, blue, thick]  (\a + 2*\f + \b + \c, 0)-- (\a + 2*\f + \b + \c + \e ,0);
\draw[<-, red, thick]   (\a + 2*\f + \b + \c + \e ,0)-- (\a + 2*\f + \b + \c + \e+\d ,0);
\draw[<-, blue, thick]   (\a + 2*\f + \b + \c + \e+\d ,0)-- (\a + 2*\f + \b + \c +2*\e+\d ,0);

\end{scope}

\begin{scope}[shift={(-2*\f - 2*\a - 2*\b, -\vs - \vss*6)}]

\fill[opacity=0.13] (-\e  - \s, -\col) rectangle (-\e + \s, \col);
\fill[opacity=0.13] (\a + \f  - \s, -\col) rectangle (\a +\f+ \s, \col);

\draw[->, red, thick] (0, 0) -- (\a, 0);
\draw[->, blue, thick] (\a, 0) -- (\a+\f, 0);
\draw[->, red, thick] (\a + \f, 0) -- (\a + \f + \b, 0);
\draw[<-, blue, thick] (\a+\f+\b, 0) -- (\a + 2*\f + \b, 0);
\draw[->, red, thick] (\a + 2*\f + \b, 0)-- (\a + 2*\f + \b + \c, 0);
\draw[->, blue, thick]  (\a + 2*\f + \b + \c, 0)-- (\a + 2*\f + \b + \c + \e ,0);
\draw[<-, red, thick]   (\a + 2*\f + \b + \c + \e ,0)-- (\a + 2*\f + \b + \c + \e+\d ,0);
\draw[<-, blue, thick]   (\a + 2*\f + \b + \c + \e+\d ,0)-- (\a + 2*\f + \b + \c +2*\e+\d ,0);

\end{scope}

\end{scope}

\begin{scope}[xscale=-1, shift={(-\a-\b-\c-\d-2*\e-2*\f, -\vs-\vss*7)}]

\fill[opacity=0.13] (\a + \f  - \s, -\col) rectangle (\a +\f+ \s, \col);

\draw[->, red, thick] (0, 0) -- (\a, 0);
\draw[->, blue, thick] (\a, 0) -- (\a+\f, 0);
\draw[->, red, thick] (\a + \f, 0) -- (\a + \f + \b, 0);
\draw[<-, blue, thick] (\a+\f+\b, 0) -- (\a + 2*\f + \b, 0);
\draw[->, red, thick] (\a + 2*\f + \b, 0)-- (\a + 2*\f + \b + \c, 0);
\draw[->, blue, thick]  (\a + 2*\f + \b + \c, 0)-- (\a + 2*\f + \b + \c + \e ,0);
\draw[<-, red, thick]   (\a + 2*\f + \b + \c + \e ,0)-- (\a + 2*\f + \b + \c + \e+\d ,0);
\draw[<-, blue, thick]   (\a + 2*\f + \b + \c + \e+\d ,0)-- (\a + 2*\f + \b + \c +2*\e+\d ,0);

\end{scope}

\end{tikzpicture}}
\caption{Schematics for $\psi = g^{a} \varphi^{f} g^{b} \varphi^{-f} g^{c} \varphi^{e} g^{d} \varphi^{-e}$ and various cyclic conjugates $w_{B}^{-1} w_{B - \Xi}$ of $\psi^{-1}$. The shadowed part marks subwords of  $w_{B}^{-1} w_{B - \Xi}$ of the form $g^{\pm K_{sep}} \varphi^{\pm K_{sep}}$ or  $\varphi^{\pm K_{sep}} g^{\pm K_{sep}}$ that fellow travel with with $\varphi^{\pm 2K_{Sep}}$.}
\label{fig:phiPowergPower}
\end{figure}
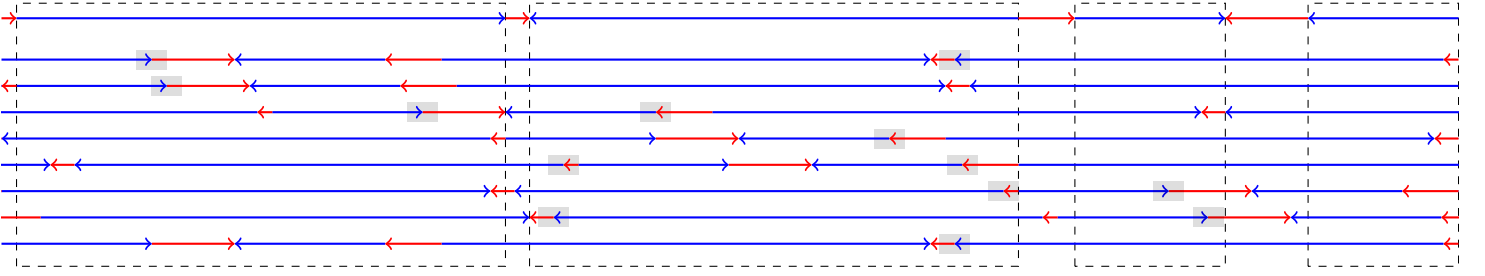

\begin{enumerate}
\item When $R \in \{0, \ldots, a\}$: in this case, we have \[\begin{aligned}
\psi &= (g^{a}) \cdot \varphi^{f} \cdot ( g^{b}\varphi^{-f} g^{c} \varphi^{e} g^{-d} \varphi^{-e}),
\\
v &= (g^{-R} \varphi^{e - K_{sep}}) \cdot \varphi^{K_{sep}}g^{d} \cdot (\varphi^{-e} g^{-c} \varphi^{f}g^{-b} \varphi^{-f } \cdot g^{R-a}).
\end{aligned}
\]
Here, note that \[\begin{aligned}
|g^{a}|_{\C*} &=_{10C}al_{g} \le R l_{g} + (e-K_{sep})l_{\varphi} =_{10C} | g^{-R} \varphi^{e-K_{sep}}|_{\C*}, \\
|g^{b}\varphi^{-f} g^{c} \varphi^{e} g^{-d} \varphi^{-e}|_{\C*} &=_{10C} (b+c+d)l_{g} + (2e+f)l_{\varphi} \\
&\le (b+c) l_{g} + (e + 2f) l_{\varphi} + (a-R) l_{g} =_{10C} |\varphi^{-e} g^{-c} \varphi^{f}g^{-b} \varphi^{-f } \cdot g^{R-a}|_{\C*}.
\end{aligned}
\]
Furthermore, we have $|\varphi^{K_{sep}}|_{\C*} \ge_{2C} 0.5C \cdot K_{Sep} \ge K_{sep}$. Likewise, for large $d \gg K_{sep}$ we have $|g^{d}|_{\C*}\ge K_{sep}$. Hence, the 4 inequalities in Display \ref{eqn:contra4Ineq} hold, a contradiction.

\item When $R \in \{a, \ldots, a + f - 2e \}$ : we have \[\begin{aligned}
\psi &= (g^{a}) \cdot \varphi^{f} \cdot ( g^{b}\varphi^{-f} g^{c} \varphi^{e} g^{-d} \varphi^{-e}),
\\
v &= (\varphi^{-(R-a)} g^{-a} \varphi^{e - K_{sep}}) \cdot \varphi^{K_{sep}}g^{d} \cdot (\varphi^{-e} g^{-c} \varphi^{f}g^{-b} \varphi^{-f + (R-a)}).
\end{aligned}
\]
Here, note that \[\begin{aligned}
|g^{a}|_{\C*} &=_{10C}al_{g} \le al_{g} + (e-K_{sep}+(R-a)) l_{\varphi} =_{10C} | \varphi^{-(R-a)}g^{-a} \varphi^{e-K_{sep}}|_{\C*}, \\
|g^{b}\varphi^{-f} g^{c} \varphi^{e} g^{-d} \varphi^{-e}|_{\C*} &=_{10C} (b+c+d)l_{g} + (2e+f)l_{\varphi} \le (b+c) l_{g} + (3e+f) l_{\varphi}\\
&\le (b+c) l_{g} + (e +2f - (R-a)) l_{\varphi}  =_{10C} |\varphi^{-e} g^{-c} \varphi^{f}g^{-b} \varphi^{-f + (R-a)}|_{\C*}.
\end{aligned}
\]
We also again have $|g^{-d}|_{\C*}, |\varphi^{K_{sep}}|_{\C*} \ge K_{sep}$, a contradiction.

\item $R \in \{a+f-2e, \ldots, a+ f\}$: we have \[\begin{aligned}
\psi &= (g^{a}  \varphi^{f}  g^{b}) \cdot \varphi^{-f} \cdot (g^{c} \varphi^{e} g^{-d} \varphi^{-e}),
\\
v &= (\varphi^{-(R-a)} g^{-a} \varphi^{e } g^{d} \cdot \varphi^{-e +K_{sep}}) \cdot \varphi^{-K_{sep}} g^{-c} \cdot (\varphi^{f}g^{-b} \varphi^{-f + (R-a)}).
\end{aligned}
\]
Note that \[
\begin{aligned}
|g^{a}\varphi^{f} g^{b}|_{\C*} &=_{10C}(a+b) l_{g} + f l_{\varphi} \le (a+d)l_{g} + (f - K_{sep})l_{\varphi} \\
&\le (a+d)l_{g} + (R-a + 2e - K_{sep}) l_{\varphi} =_{10C} |\varphi^{-(R-a)} g^{-a} \varphi^{e } g^{d} \cdot \varphi^{-e +K_{sep}} |_{\C*}, \\
|g^{c} \varphi^{e} g^{-d} \varphi^{-e}|_{\C*} &=_{10C} (c+d)l_{g} + 2e l_{\varphi} \le f l_{g} \\
&\le fl_{g} + (2f - R + a)  l_{\varphi} =_{10C} |\varphi^{f}g^{-b} \varphi^{-f + (R-a)}|_{\C*}.
\end{aligned}
\]
Further, $|\varphi^{-K_{sep}}|_{\C*}, |g^{-c}|_{\C*} \ge K_{sep}$. This contradicts Display \ref{eqn:contra4Ineq}.

\item $R \in \{a+f, \ldots, a+ f+b\}$: we have \[\begin{aligned}
\psi &= (g^{a}  \varphi^{f}  g^{b}) \cdot \varphi^{-f} \cdot (g^{c} \varphi^{e} g^{-d} \varphi^{-e}),
\\
v &= (g^{-(R - a - f)}\varphi^{-f} g^{-a} \varphi^{e } g^{d} \cdot \varphi^{-e +K_{sep}}) \cdot \varphi^{-K_{sep}} g^{-c} \cdot (\varphi^{f}g^{-b + (R-a-f)}).
\end{aligned}
\]
Note that \[
\begin{aligned}
|g^{a}\varphi^{f} g^{b}|_{\C*} &=_{10C}(a+b) l_{g} + f l_{\varphi} \le (a+d)l_{g} + (f - K_{sep})l_{\varphi} \\
&\le (a+d + (R-a-f))l_{g} + (2e  + f - K_{sep}) l_{\varphi} \\
&=_{10C} |g^{-(R - a - f)}\varphi^{-f} g^{-a} \varphi^{e } g^{d} \cdot \varphi^{-e +K_{sep}} |_{\C*}, \\
|g^{c} \varphi^{e} g^{-d} \varphi^{-e}|_{\C*} &=_{10C} (c+d)l_{g} + 2e l_{\varphi} \le f l_{g} \\
&\le fl_{g} + (2f - R + a)  l_{\varphi} =_{10C} |\varphi^{f}g^{-b + (R-a-f)}|_{\C*}.
\end{aligned}
\]
Further, $|\varphi^{-K_{sep}}|_{\C*}, |g^{-c}|_{\C*} \ge K_{sep}$. This contradicts Display \ref{eqn:contra4Ineq}.

\item $R \in \{a+f+b, \ldots, a+ f+2b\}$: we have \[\begin{aligned}
\psi &= (g^{a}  \varphi^{f}  g^{b}) \cdot \varphi^{-f} \cdot (g^{c} \varphi^{e} g^{-d} \varphi^{-e}),
\\
v &= (\varphi^{R - (a+f+b)}g^{-b}\varphi^{-f} g^{-a} \varphi^{e } g^{d}  \varphi^{-e +K_{sep}}) \cdot \varphi^{-K_{sep}} g^{-c} \cdot (\varphi^{a+b+2f - R}).
\end{aligned}
\]
Note that \[
\begin{aligned}
|g^{a}\varphi^{f} g^{b}|_{\C*} &=_{10C}(a+b) l_{g} + f l_{\varphi} \le (a+d)l_{g} + (f - K_{sep})l_{\varphi} \\
&\le (a+b+d)l_{g} + (2e  + R -a-b - K_{sep} ) l_{\varphi} \\
&=_{10C} |\varphi^{R - (a+f+b)}g^{-b}\varphi^{-f} g^{-a} \varphi^{e } g^{d}  \varphi^{-e +K_{sep}} |_{\C*}, \\
|g^{c} \varphi^{e} g^{-d} \varphi^{-e}|_{\C*} &=_{10C} (c+d)l_{g} + 2e l_{\varphi} \le (f-b) l_{g} \\
&\le (a+b+2f - R)  l_{\varphi} =_{10C} |\varphi^{a+b+2f - R}|_{\C*}.
\end{aligned}
\]
Further, $|\varphi^{-K_{sep}}|_{\C*}, |g^{-c}|_{\C*} \ge K_{sep}$. This contradicts Display \ref{eqn:contra4Ineq}.

\item $R \in \{a+f+2b, \ldots, a+ 2f+b\}$: we have \[\begin{aligned}
\psi &= (g^{a}  \varphi^{f}  g^{b}) \cdot \varphi^{-f} \cdot (g^{c} \varphi^{e} g^{-d} \varphi^{-e}),
\\
v &= (\varphi^{R - (a+f+b)}g^{-b}\varphi^{-f + K_{sep} }) \cdot \varphi^{-K_{sep}}g^{-a} \cdot (\varphi^{e} g^{d} \varphi^{-e } g^{-c} \varphi^{a+b+2f - R}).
\end{aligned}
\]
Note that \[
\begin{aligned}
|g^{a}\varphi^{f} g^{b}|_{\C*} &=_{10C}(a+b) l_{g} + f l_{\varphi} \le bl_{g} + (f +b - K_{sep})l_{\varphi} \\
&\le bl_{g} + (R  - a- b - K_{sep}) l_{\varphi} =_{10C} |\varphi^{R - (a+f+b)}g^{-b}\varphi^{-f + K_{sep} } |_{\C*}, \\
|g^{c} \varphi^{e} g^{-d} \varphi^{-e}|_{\C*} &=_{10C} (c+d)l_{g} + 2e l_{\varphi}  \le (c+d)l_{g} + (2e+ (R - a-f-2b)) l_{\varphi}\\
&=_{10C} | \varphi^{e} g^{d} \varphi^{-e } g^{-c} \varphi^{a+b+2f - R}|_{\C*}.
\end{aligned}
\]
Further, $|\varphi^{-K_{sep}}|_{\C*}, |g^{-a}|_{\C*} \ge K_{sep}$. This contradicts Display \ref{eqn:contra4Ineq}.

\item $R \in \{a+b+2f, \ldots, a+ b+ 2f + c/2\}$: we have \[\begin{aligned}
\psi &= (g^{a}  \varphi^{f}  g^{b}\varphi^{-f} g^{c}) \cdot \varphi^{e} \cdot (g^{-d} \varphi^{-e}),
\\
v &= (g^{- R +a+b+2f}\varphi^{f}g^{-b}\varphi^{-f}g^{-a} \varphi^{e-K_{sep}})\cdot \varphi^{K_{sep}} g^{K_{sep}} \cdot ( g^{d-K_{sep}} \varphi^{-e } g^{R - a-b-2f - c} ).
\end{aligned}
\]
Note that \[
\begin{aligned}
|g^{a}  \varphi^{f}  g^{b}\varphi^{-f} g^{c}|_{\C*} &=_{10C}(a+b+c) l_{g} + 2f l_{\varphi} \le (a+b)l_{g} + (2f + e - K_{Sep})l_{\varphi} \\
&\le ( R - 2f) l_{g} + (2f + e - K_{sep}) l_{\varphi}=_{10C} |g^{- R +a+b+2f}\varphi^{f}g^{-b}\varphi^{-f}g^{-a} \varphi^{e-K_{sep}} |_{\C*}, \\
|g^{-d} \varphi^{-e}|_{\C*} &=_{10C} dl_{g} + e l_{\varphi}  \le (c/2 + d-K_{sep})l_{g} + e l_{\varphi}\\
&\le (d-K_{sep} + (a+b+2f + c) - R) l_{g} + e l_{\varphi} =_{10C} | g^{d-K_{sep}} \varphi^{-e } g^{R - a-b-2f - c}|_{\C*}.
\end{aligned}
\]
Further, $|\varphi^{K_{sep}}|_{\C*}, |g^{K_{sep}}|_{\C*} \ge K_{sep}$. This contradicts Display \ref{eqn:contra4Ineq}.

\item $R \in \{a+b+2f+c/2 \ldots, a+ b+ 2f + c\}$: we have \[\begin{aligned}
\psi &= (g^{a}  \varphi^{f}  g^{b}) \cdot \varphi^{-f}\cdot ( g^{c} \varphi^{e} g^{-d} \varphi^{-e}),
\\
v &= (g^{- R +a+b+2f}\varphi^{f}) \cdot g^{-b}\varphi^{-K_{sep}} \cdot (\varphi^{-f + K_{sep}}g^{-a} \varphi^{e} g^{d} \varphi^{-e } g^{R - a-b-2f - c} ).
\end{aligned}
\]
Note that \[
\begin{aligned}
|g^{a}\varphi^{f} g^{b}|_{\C*} &=_{10C}(a+b) l_{g} + f l_{\varphi} \le 0.5 c l_{g} + fl_{\varphi} \\
&\le ( R - a-b-2f) l_{g} + f l_{\varphi}=_{10C} |g^{- R +a+b+2f}\varphi^{f} |_{\C*}, \\
|g^{c} \varphi^{e} g^{-d} \varphi^{-e}|_{\C*} &=_{10C} (c+d)l_{g} + 2e l_{\varphi}  \le (a+d)l_{g} + (2e+ f-K_{sep}) l_{\varphi}\\
&\le ((R-a-b-2f-c) + a+d) l_{g} + (2e + f - K_{sep}) l_{\varphi} \\
&=_{10C} | \varphi^{-f + K_{sep}}g^{-a} \varphi^{e} g^{d} \varphi^{-e } g^{R - a-b-2f - c} |_{\C*}.
\end{aligned}
\]
Further, $|\varphi^{-K_{sep}}|_{\C*}, |g^{-b}|_{\C*} \ge K_{sep}$. This contradicts Display \ref{eqn:contra4Ineq}.

\item Remaining cases, i.e., $R \in \{a+ b+ 2f + c, \ldots, a+b+c+d+2e+2f\}$: in these cases, we again use the subword decomposition $\psi= (g^{a}  \varphi^{f}  g^{b}) \cdot \varphi^{-f}\cdot ( g^{c} \varphi^{e} g^{-d} \varphi^{-e})$. The word $v$ has decomposition $v = v_f \cdot g^{-b}\varphi^{-K_{sep}} \cdot v_{b}$, where \begin{itemize}
\item $v_{f} = \varphi^{-k} g^{-c}\varphi^{f}$ and $v_{b}= \varphi^{-f + K_{sep}}g^{-a}\varphi^{e} g^{d} \varphi^{-(e-k)}$ for some $0 \le k \le e$, or
\item $v_{f} =g^{k}\varphi^{-e} g^{-c}\varphi^{f}$ and $v_{b}= \varphi^{-f + K_{sep}}g^{-a} \varphi^{e} g^{d-k}$ for some $0 \le k \le d$, or
\item $v_{f}=\varphi^{k}g^{d}\varphi^{-e} g^{-c}\varphi^{f}$ and $v_{b}=\varphi^{-f + K_{sep}}g^{-a} \varphi^{e-k}$ for some $0 \le k \le e$.
\end{itemize}
In each of these cases, we have \[
|v_{f}|_{\C*} \ge | g^{-c}\varphi^{f}|_{\C*} =_{10C} cl_{g} + fl_{\varphi} \ge (a+b)l_{g} + fl_{g} =_{10C} |g^{a} \varphi^{f} g^{b}|_{\C*}.
\]
Moreover, we have \[
|v_{b}|_{\C*} \ge |\varphi^{-f + K_{sep}}g^{-a}|_{\C*} =_{10C} a l_{g} +  (f-K_{sep})l_{\varphi} \ge (c+d)l_{g} + 2el_{\varphi} = |g^{c} \varphi^{e} g^{-d} \varphi^{-e}|_{\C*} .
\]Further, $|\varphi^{-K_{sep}}|_{\C*}, |g^{-b}|_{\C*} \ge K_{sep}$. This contradicts Display \ref{eqn:contra4Ineq}.
\end{enumerate}

All in all, there is no $\phi_{T} \in \Homeo(S)$ realizing Display \ref{eqn:contraT} for $T = 20(K+C)/\Xi C + 100$. Hence, $\psi \not\sim \psi^{-1}$.
\end{proof}

Let us now consider $1 \ll a \ll b \ll c \le d \ll e \ll f \ll a' \ll b' \ll c' \le d' \ll e' \ll f'$ and construct  \[
\psi = g^{a} \cdot \varphi^{f} g^{b} \varphi^{-f} \cdot g^{c} \cdot \varphi^{e} g^{-d} \varphi^{-e}, \quad \psi' = g^{a'} \cdot \varphi^{f'} g^{b'} \varphi^{-f'} \cdot g^{c'} \cdot \varphi^{e'} g^{-d'} \varphi^{-e'}.
\]
Let $(p_{i})_{i \in \Z}$ be the $\{g, \varphi\}$-word sequence for $\psi$, i.e., $p_{0} = x_{0}, p_{1} = g x_{0}, \ldots, p_{a+1} = g^{a} \varphi x_{0}, \ldots, p_{a+b+c+d+2e+2f} = \psi x_{0}$, and $p_{k(a+b+c+d+2e+2f) + i}= \psi^{k} p_{i}$. Similarly define the $\{g, \varphi\}$-word sequence $(q_{i})_{i \in \Z}$ for $\psi'$. These sequences are uniform quasi-geodesics.

Suppose to the contrary that arbitrarily long subsequence of  $(q_{i})_{i \in \Z}$ is fellow traveled by a $\Homeo(S)$-translate of a subsequence of $(q_{i})_{i \in \Z}$. In particular, suppose that there exists $\phi \in \Homeo(S)$ such that $(q_{a'}, \ldots, q_{a'+f'})$ and $(\phi p_{i}, \ldots, \phi p_{j})$ are $10C$-fellow traveling for some $i<j$. Since $d_{\C*(S)}(p_{i}, p_{j})$ is at most $0.9C|j-i|$, $j-i$ is much larger than $a+b+c+d+2e+2f$ (since $f' \gg a, b, c, d, e, f$). This implies that there is a full period for $\psi$ between $i$ and $j$, of the form $\{ M (a+b+c+d+2e+2f), \ldots, (M+1) (a+b+c+d+ 2e+2f)\}$. Hence, there exists $a' < k < m < a'+f'$ such that \[
d_{\C*(S)} (\varphi^{k} x_{0}, \phi \psi^{M} g^{a} \cdot x_{0}) < 10C, \quad 
d_{\C*(S)} (\varphi^{m} x_{0}, \phi \psi^{M} g^{a} \cdot \varphi^{f} \cdot g^{b} x_{0}) < 10C.
\]
This cannot happen for large enough $f$ and $b$, as $(\ldots, \phi^{-2}, \phi^{-1}, id, g, g^{2}, \ldots) \not\sim (\varphi^{l})_{l \in \Z}$. In conclusion, $\psi \not \sim \psi'$.

Note also that every long enough subsequence of $(q_{-i} )_{i \in \Z}$ (the reversal of $(q_{i})_{i \in \Z}$) contains a translate of $(x_{0}, \varphi x_{0}, \ldots, \varphi^{f'} x_{0})$ as well. For a similar reason, we obtain $\psi^{-1} \not \sim \psi'$, given that $1 \ll a \ll b \ll \ldots \ll f \ll a' \ll b' \ll \ldots \ll f'$.

To summarize, we have: \begin{claim}\label{claim:psiInd}
For \[\begin{aligned}
1 &\ll a_{1} \ll b_{1} \ll c_{1} \le d_{1} \ll e_{1} \ll  f_{1} \\
&\ll a_{2} \ll b_{2} \ll c_{2} \le d_{2} \ll e_{2} \ll  f_{2} \ll a_{3} \ll \ldots,
\end{aligned}
\] the homeomorphisms $\psi_{i} := \psi_{a_{i}, b_{i}, c_{i}, d_{i}, e_{i}, f_{i}} \in \Homeo(S)$ satisfy that: \begin{enumerate}
\item $\psi_{i}$ fixes an open neighborhood of $p$ for each $i$;
\item $\psi_{i} \not\sim \psi_{j}^{\pm}$ for each $i<j$, and 
\item $\psi_{i} \not \sim \psi_{i}^{-1}$ for each $i$.
\end{enumerate}
\end{claim}

Now let $D$ be a closed disc embedded in $S$. One can conjugate $\psi_{i}$ to a homeomorphism $\Psi_{i}$ (in $\Homeo(S)$) that fixes $D$, by a conjugator that sends $D$ into the fixed point set of $\psi_{i}$, a neighborhood of $p$. It follows that $\psi_{i}^{\pm 1} \sim \Psi_{i}^{\pm 1}$. Since $\sim$ is a transitive relation, we conclude:

\begin{prop}\label{prop:DUnifInd}
Let $D$ be a closed disc embedded in $S$. Then there exist infinitely many homeomorphisms $\Psi_{1}, \Psi_{2}, \ldots \in \Homeo(S)$ such that: \begin{enumerate}
\item $\Psi_{i}$ fixes $D$ for each $i$;
\item $\Psi_{i} \not\sim \Psi_{j}^{\pm}$ for each $i<j$, and 
\item $\Psi_{i} \not \sim \Psi_{i}^{-1}$ for each $i$.
\end{enumerate}
\end{prop}

As observed by Brooks \cite{MR624804} and generalized by Bestvina and Fujiwara for WPD actions on $\delta$-hyperbolic spaces \cite{bestvina2002bounded}, one can construct countably infinitely many linearly independent quasi-morphisms $Q_{1}, Q_{2}, \ldots$ that are unbounded on $Fix_{D}$. Namely, for each $i$, there exists a Brooks quasi-morphism that assigns nonzero value on some power of $\Psi_{i}$, but assigns zero value on any power of $\Psi_{j}$ for $j \neq i$. Since $\Psi_{i}^{k} \in Fix_{D}$ for each $k$, we conclude:

\begin{cor}\label{cor:inf}
Let $S = \mathbb{T}^{2}$ be the torus and let $D$ be an embedded closed disc on $S$. Let $Fix_{D} \le \Homeo(S)$ be the group of homeomorphisms that preserve $D$ pointwise. Then the space $\widetilde{QH}(\Homeo(S); Fix_{D} \cap \Homeo_{0}(S))$ of quasi-morphisms on $\Homeo(S)$ that is unbounded on $Fix_{D} \cap \Homeo_{0}(S)$ is infinite dimensional.
\end{cor}

%
%

\medskip
\bibliographystyle{alpha}
\bibliography{wpd}

\end{document}